\documentclass[a4paper]{article}

\usepackage{graphicx} 
\usepackage{amsthm, amsmath, amssymb, amsfonts}
\usepackage{latexsym}
\usepackage{mathrsfs} 
\usepackage{mathtools}
\usepackage{url}
\usepackage{enumerate}
\usepackage[margin=20mm]{geometry}
\usepackage{hyperref}
\usepackage[dvipsnames]{xcolor}
\usepackage{stmaryrd} 
\usepackage{docmute}
\usepackage{centernot}

\definecolor{OLgreen}{rgb}{0.074,0.541,0.027}

\mathtoolsset{showonlyrefs}

\allowdisplaybreaks

\hypersetup{
    colorlinks=true,
    linkcolor={blue!80!black},
    citecolor={OLgreen!90!black},
    urlcolor=magenta,
    linktocpage=true,
}

\numberwithin{equation}{section}

\theoremstyle{plain}
\newtheorem{Thm}{Theorem}[section]
\newtheorem{Lem}[Thm]{Lemma}
\newtheorem{Prp}[Thm]{Proposition}

\theoremstyle{definition}

\newtheorem{Ass}[Thm]{Assumption}

\newtheorem{Rem}[Thm]{Remark}

\theoremstyle{plain}
\newtheorem*{Thm*}{Theorem}
\newtheorem*{Lem*}{Lemma}
\newtheorem*{Prp*}{Proposition}
\newtheorem*{Cor*}{Corollary}
\theoremstyle{definition}
\newtheorem*{Def*}{Definition}
\newtheorem*{Rem*}{Remark}
\newtheorem*{Exa*}{Example}

\DeclarePairedDelimiter{\norm}{\lVert}{\rVert} 
\DeclarePairedDelimiter{\abra}{\langle}{\rangle} 

\newcommand{\E}{\mathbb{E}}
\newcommand{\N}{\mathbb{N}}
\newcommand{\R}{\mathbb{R}}

\newcommand{\bB}{\mathbb{B}}

\newcommand{\bD}{\mathbb{D}}

\newcommand{\bF}{\mathbb{F}}

\newcommand{\bT}{\mathbb{T}}

\newcommand{\cB}{\mathcal{B}}

\newcommand{\cF}{\mathcal{F}}
\newcommand{\cG}{\mathcal{G}}
\newcommand{\cH}{\mathcal{H}}
\newcommand{\cI}{\mathcal{I}}
\newcommand{\cJ}{\mathcal{J}}

\newcommand{\cL}{\mathcal{L}}

\newcommand{\cO}{\mathcal{O}}

\newcommand{\cS}{\mathcal{S}}

\newcommand{\cX}{\mathcal{X}}

\newcommand{\sF}{\mathscr{F}}

\newcommand{\ind}{\mathbf{1}}

\newcommand{\dis}{\displaystyle}

\newcommand{\relmiddle}[1]{\mathrel{}\middle#1\mathrel{}}

\title{\texorpdfstring{\vspace{-8mm}}{}Large deviations for a spatial average of \\ stochastic heat and wave equations}
\author{Masahisa Ebina
\thanks{Institute of Mathematics for Industry, Kyushu University, 744 Motooka, Nishi-ku, Fukuoka 819-0395, Japan. \\ E-mail: \texttt{ebinamasahisa@gmail.com}}
}
\date{}

\begin{document}
\maketitle
\begin{abstract}
We consider the one-dimensional stochastic heat and wave equations driven by Gaussian noises with constant initial conditions. 
We study the spatial average of the solutions on an interval of length $R$ and show that the family of laws of the spatial average satisfies the large deviation principle as $R$ goes to infinity.
We also present the large deviation principle in the space of continuous functions. 
We prove these results using the tools of Malliavin calculus to evaluate the covariance of nonlinear functionals of the solution.
\end{abstract}


\maketitle

\section{Introduction}\label{section Introduction}
This paper considers the following form of one-dimensional stochastic heat and wave equations
\begin{align}
\label{SPDE}    
    \mathcal{L}u(t,x) = \sigma(u(t,x))\dot{W}(t,x), \quad (t,x) \in [0,T] \times \mathbb{R}, \quad  \cL = \partial_t - \frac{1}{2}\partial_x^2 \ \ \text{or} \ \ \cL = \partial_t^2 - \partial_x^2,
\end{align}
with constant initial conditions.
Here $T > 0$ is fixed, $\sigma\colon \mathbb{R} \to \mathbb{R}$, and $\dot{W}(t,x)$ is the formal notation of a centered Gaussian noise.
We analyze a random field solution $\{u(t,x)\}_{(t,x) \in [0,T] \times \R}$ to \eqref{SPDE} (see Section \ref{subsection Solutions to stochastic heat and wave equations} for details), and the goal of this paper is to show the Large Deviation Principle (LDP for short) for the spatial average of the solution process defined by
\begin{equation*}
    \frac{1}{R}F_R(t) \coloneqq \frac{1}{R}\int_{0}^{R}(u(t,x) - \E[u(t,x)])dx, \quad t\in [0,T]. 
\end{equation*}

In recent years, the asymptotic behavior of spatial averages of solutions to stochastic PDEs has been actively studied, and various limit theorems have been established in different settings using tools from Malliavin calculus.
For example, the works \cite{1dSHECLT, SHECLT, 1dSWECLT, SWE12d} treat the equation \eqref{SPDE} and establish quantitative central limit theorems and its functional version for the spatial averages of the solution using Malliavin-Stein's method.
Following these results, the authors of \cite{SHEergo, d123SWEergo} investigate in detail the spatial ergodicity of the solutions to stochastic heat and wave equations and show that the law of large numbers
\begin{equation}\label{LLN}
    \frac{1}{R}F_R(t) \xrightarrow[]{R \to \infty} 0, \qquad \text{a.s. and in $L^p(\Omega)$, $p \in [1,\infty)$} 
\end{equation}
holds under some mild conditions on noises.
Furthermore, the law of the iterated logarithm for $F_R(t)$, when $\cL$ is the heat operator, is also shown in \cite{1dSHELIL}.

The common underlying idea in all of these works is that the spatial average of the solution can be approximately viewed as a sum of weakly dependent random variables, depending on the spatial correlation of the driving noise.
With this and \eqref{LLN} in mind, it is then natural to study the associated LDP for the law of spatial average $R^{-1}F_R(t)$.
However, it remains unknown whether the LDP holds, except in the trivial case where $\sigma$ is a constant function, in which the law of $F_R(t)$ follows a Gaussian distribution.
For general $\sigma$, a large deviation upper bound can be obtained if $\sigma$ is a bounded Lipschitz function, using a similar argument as in \cite[Lemma 2.1]{MC2}.
However, this argument cannot provide a large deviation lower bound and, therefore, cannot establish the LDP.
This paper aims to address this issue and to present that the LDP holds for the family of the laws of $\{R^{-1}F_R(t)\}_{R > 0}$ under stronger conditions on $\sigma$.

\subsection{Settings and main results}

Throughout the paper, we assume that a centered Gaussian noise $\dot{W}(t,x)$ in \eqref{SPDE} has the following covariance structure
\begin{equation}\label{noise cov}
    \E[\dot{W}(t,x)\dot{W}(s,y)] = \delta_0(t-s)\Gamma(x-y),
\end{equation}
where $\delta_0$ is the Dirac delta function and $\Gamma$ is a nonnegative, nonnegative definite tempered measure. 

To establish the LDP results, we need to impose certain conditions on $\sigma$ in \eqref{SPDE} and on $\Gamma$.

\begin{Ass}
\label{Assum 1}
We work under the following assumptions:
\begin{enumerate}
    \item[(1)] $\sigma \colon \R \to \R$ belongs to $C_{\mathrm{b}}^1(\R)$, the space of bounded $C^1$ functions with bounded derivatives.
    \item[(2)] $\Gamma$ satisfies one of the following two cases:
    \begin{enumerate}
        \item[(I)] $\Gamma(dx) = \delta_0(dx)$, \quad \textit{i.e.} $\dot{W}$ is the space-time white noise.
        \item[(II)] 
        $\Gamma(dx) = \Gamma(x)dx$ is a nonnegative, nonnegative definite measure, and for some $\rho>0$ and $\eta > 1$,  
        \begin{equation}
            \Gamma(x) = O(\exp(-\rho|x|^{\eta})), \quad \text{as $|x| \to \infty$.}
        \end{equation}
    \end{enumerate}
\end{enumerate}
\end{Ass}

\begin{Rem} 
\begin{enumerate}
    \item[(1)] For technical reasons, we assume that $\sigma$ is bounded so that the quadratic variations associated with the solution $u$ can be controlled almost surely. See Section \ref{section LDP for the finite-dimensional distributions} for details.
    \item[(2)] It is known that the LDP holds for stationary sequences of random variables under sufficiently strong mixing-type conditions (see \textit{e.g.} \cite{MixingLDP}, \cite[Section 6.4]{Dembo-Zeitouni}). Given this, and recalling that the spatial integral of the solution can heuristically be interpreted as a sum of weakly dependent random variables, it is natural to impose an appropriate mixing-type condition in order to establish the LDP for the spatial average. The assumption on the spatial correlation $\Gamma$ of the noise is introduced to ensure such a mixing-type condition (cf. Proposition \ref{Prp Cov estimate}). 
    \item[(3)] In case (II), $\Gamma(x)dx$ becomes a tempered measure, and consequently, $\Gamma(x) \in L^1(\R)$. 
    A function $\Gamma$ satisfying (II) can be easily constructed using convolution. Specifically, for a  locally integrable nonnegative function $f$ with $f(x) = O(\exp(-\rho 2^{\eta} |x|^{\eta}))$, the function $\Gamma(x) = (f(-\cdot) \ast f(\cdot))(x)$ satisfies assumption (II) (\textit{cf.} Lemma \ref{Lem exp decay convolution}).
\end{enumerate}
\end{Rem}

For any $k \in \N$ and $(t_1, \ldots , t_k) \eqqcolon \mathbb{T}_k \in [0,T]^k$, let $\mu_R^{\mathbb{T}_k}$ denote the law on $\R^k$ of 
\begin{equation*}
    \frac{1}{R}\mathbb{F}_R(\mathbb{T}_k) \coloneqq \frac{1}{R}(F_R(t_1), \ldots, F_R(t_k)) = \left(\frac{1}{R}F_R(t_1), \ldots, \frac{1}{R}F_R(t_k)\right).
\end{equation*}
The first main result is the LDP for the finite-dimensional distributions of the spatial average.

\begin{Thm}
\label{Thm main 1}
Suppose that Assumption \ref{Assum 1} is satisfied.
Then, for any $k \in \N$ and $(t_1, \ldots, t_k) \eqqcolon \mathbb{T}_k \in [0,T]^k$, the family of measures $\{\mu_R^{\mathbb{T}_k}\}_{R>0}$ satisfies the LDP with speed $R$ and the convex good rate function 
\begin{equation}
\label{rate function I}
    I^{\mathbb{T}_k}(x) = \sup_{\lambda \in \R^k} \{ \lambda \cdot x - \Lambda^{\mathbb{T}_k}(\lambda )\}, \qquad \text{where} \ \  \Lambda^{\mathbb{T}_k}(\lambda) \coloneqq \lim_{R \to \infty}\frac{1}{R}\log \E[e^{\lambda \cdot \mathbb{F}_R(\mathbb{T}_k)}].
\end{equation}
In particular, the limit $\Lambda^{\mathbb{T}_k}(\lambda)$ exists for any $\lambda \in \R^k$. 
\end{Thm}

As we will see in Lemma \ref{Lem diff est}, for each $R>0$, $\{R^{-1}F_R(t)\}_{t \in [0,T]}$ admits a continuous modification. 
Let $\mu_R$ denote the law of this continuous modification. 
By applying Theorem \ref{Thm main 1}, we can extend the LDP to the space of continuous functions $C([0,T])$, which is our second main result.

\begin{Thm}
\label{Thm main 2}
Suppose that Assumption \ref{Assum 1} is satisfied.
The family of measures $\{\mu_R\}_{R>0}$ satisfies the LDP in $C([0,T])$ endowed with the uniform convergence topology, with speed $R$ and the convex good rate function 
\begin{equation*}
    I(f) = \sup_{k\in \N} \, \sup_{0 \leq t_1 < \cdots < t_k \leq T} \bigg\{ I^{\mathbb{T}_k} \Big( (f(t_1), \ldots, f(t_k)) \Big) \bigg\}, \quad f \in C([0,T]). 
\end{equation*}

\end{Thm}

\begin{Rem}
A similar LDP also holds for the laws of a spatial average $(2R)^{-1}\int_{-R}^{R}(u(t,x) - \E[u(t,x)])dx$ because, in our setting, a solution process $\{u(t,x)\}_{(t,x) \in [0,T] \times \R}$ to \eqref{SPDE} satisfies the property (S). 
See Section \ref{subsection Solutions to stochastic heat and wave equations} and Lemma \ref{Lem shift invariance} for more details. 
\end{Rem}

\begin{Rem}
Since the law of large numbers and central limit theorems have also been established for the spatial average of stochastic heat and wave equations in two or more spatial dimensions (see \textit{e.g.} \cite{SHEergo, d123SWEergo, SHECLT, SWE12d, 3dSWECLT, Ebi25}), we can expect that the LDP should hold in the higher-dimensional case under the same assumptions. 
Unfortunately, the approach used to prove Theorems \ref{Thm main 1} and \ref{Thm main 2} does not appear to apply easily to the higher-dimensional case.
To establish the LDP results, we rely on an approximate subadditivity argument to ensure the existence of certain limits in Section \ref{section LDP for the finite-dimensional distributions}. However, some of the estimates provided there become too crude to guarantee the existence of such limits, essentially due to the absence of a natural ordering in higher dimensions. 
\end{Rem}

After introducing some notation and conventions below, the structure of the paper is as follows. 
In Section \ref{section Preliminaries}, we briefly present the formulation of \eqref{SPDE} and recall some preliminary results on the large deviation theory and the Malliavin calculus used in the paper. 
Section \ref{section Covariance estimates} presents a key estimate that is crucial for proving the main results. Then, we prove Theorem \ref{Thm main 1} and Theorem \ref{Thm main 2} in Section \ref{section LDP for the finite-dimensional distributions} and Section \ref{section Sample path LDP}, respectively.

\vspace{2mm}
\noindent
\textbf{Acknowledgment.}
The author would like to thank his advisor, Professor Seiichiro Kusuoka, for his encouragement and several valuable comments on an earlier version of this work.
He also thanks Dr. Takahiro Mori for his helpful suggestions.
This work was supported by the Japan Society for the Promotion of Science (JSPS), KAKENHI Grant Number JP22J21604.

\vspace{2mm}
\noindent
\textbf{Notation.} 
We will use the following standard notations throughout the paper. 
\begin{itemize}
    \item For $x,y \in \R^k$, let $|x|$ and $x \cdot y$ denote the usual Euclidean norm and inner product on $\R^k$, respectively. 
    \item For $a, b \in \R$, we use $a \land b \coloneqq \min\{a, b\}$ and $a \lor b \coloneqq \max\{a,b\}$.
    \item For a Lipschitz function $f$, we write $\mathrm{Lip}(f)$ for its Lipschitz constant. 
    \item For some $l$ quantities $Q_1, \ldots Q_l$ \ $(l \in \N)$,  we denote by $C_{Q_1, \ldots, Q_l}$ a constant depending on $Q_1, \ldots Q_l$ that may vary from line to line. 
    These constants are not necessarily the same in different places and are used to simplify expressions.
\end{itemize}

In this paper, we adopt the following notational convention (with slight abuse of notation) to treat both cases of $\Gamma$ in Assumption \ref{Assum 1} in a unified manner.

\vspace{2mm}
\noindent
\textbf{Convention.} 
Throughout the paper, we set
\begin{equation*}
    \lVert \delta_0 \rVert_{L^1(\R)} \coloneqq 1 \quad \text{and} \quad \int_{\R^2} f(x)g(y) \delta_{0}(x-y)dxdy \coloneqq \int_{\R}f(x)g(x)dx. 
\end{equation*}
When $\Gamma = \delta_0$, we interpret $\delta_0$ as $\delta_0(x) = O(\exp(-\rho|x|^{\eta}))$ with $\eta = \infty$, in analogy with case (II) of Assumption \ref{Assum 1}. 
For example, expressions like $2 \land \eta$ that appear in later sections should be understood as $2 \land \infty = 2$ when $\Gamma = \delta_0$. 
On the other hand, for constants that are stated to depend on $\eta$ but are not explicitly written, we assume that they are independent of $\eta$ when $\Gamma = \delta_0$. 
For example, a constant like $C_{T, \eta, \rho}$ that appears in later sections should be understood as not depending on $\eta$ when $\Gamma = \delta_0$.

\section{Preliminaries}\label{section Preliminaries}

\subsection{Stochastic heat and wave equations}\label{subsection Solutions to stochastic heat and wave equations}
In this section, we recall the formulation of \eqref{SPDE} using the random field approach. 
Let us define the separable real Hilbert space $\cH$ by completing the space $\cS(\R)$ of real-valued rapidly decreasing functions on $\R$ with the real inner product
\begin{equation}
    \abra{\varphi, \psi}_{\mathcal{H}} = \int_{\R^2}\varphi(x)\psi(y)\Gamma(x-y)dxdy, \qquad \varphi, \psi \in \cS(\R). 
\end{equation}
We set $\cH_T = L^2([0,T]; \cH)$. 
When $\Gamma = \delta_0$, then $\cH = L^2(\R)$ and $\cH_T = L^2([0,T];L^2(\R)) \simeq L^2([0,T] \times \R)$.
Note that $\ind_{[0,t]}\ind_{A} \in \cH_T$ for all $t \geq 0$ and $A \in \cB_{\mathrm{b}}(\R)$, where $\cB_{\mathrm{b}}(\R)$ is the collection of all bounded Borel sets of $\R$.

Let $W = \{W(h) \mid h \in \cH_T \}$ be a centered Gaussian process with the covariance function 
\begin{equation}
    \E[W(h_1)W(h_2)] = \abra{h_1, h_2}_{\cH_T},
\end{equation}
defined on a complete probability space $(\Omega, \sF, P)$, where $\sF$ is assumed to be generated by $W$. 
We write $\sF^0_t$ for the $\sigma$-algebra generated by $\{W(\ind_{[0,s]}\ind_{A}) \mid 0 \leq s \leq t, \ A \in \cB_{\mathrm{b}}(\R) \}$ and the $P$-null sets, and we set $\sF_t = \cap_{s>t}\sF_s^0$ for $t \in [0,T)$ and $\sF_T = \sF_T^0$.
Then, we can construct a worthy martingale measure $\{W(\ind_{[0,t]}\ind_{A}) \mid t \in [0,T], A \in \cB_{\mathrm{b}}(\R)\}$ with respect to $\{\sF_t\}_{t \in [0,T]}$. 
With this measure, we can apply the theory of stochastic integration developed in \cite{Walsh} to formulate the equation \eqref{SPDE}. 
For further details, we refer the reader to \cite{Walsh, DQ}.  

We consider \eqref{SPDE} with the constant initial conditions 
\begin{equation}
    u(0,x) = c_H \ \  \text{for $\cL = \partial_t - \frac{1}{2}\partial_x^2$} \qquad \text{and} \qquad u(0,x) = c_{W1}, \  \partial_t u(0,x) = c_{W2} \ \  \text{for $\cL = \partial_t^2 - \partial_x^2$}.
\end{equation}
A solution to \eqref{SPDE} is defined as a predictable process $\{u(t,x)\}_{(t,x) \in [0,T] \times \R}$ satisfying
\begin{equation}\label{sol eq}
    u(t,x) = \cI(t) + \int_0^t\int_{\R}G(t-s, x-y)\sigma(u(s,y))W(ds,dy) \quad \text{a.s.}, 
\end{equation}
for all $(t,x) \in [0,T] \times \R$. 
Here, $\cI(t)$ is the contribution from initial conditions and $G$ is the associated fundamental solution of the operator $\cL$ in \eqref{SPDE}.
In particular, when $\cL = \partial_t - \frac{1}{2}\partial_x^2$,
\begin{equation}\label{cond she}
    G(t,x) = (2\pi t)^{-\frac{1}{2}}e^{-\frac{x^2}{2t}}, \quad \cI(t) = G(t) \ast c_H = c_H,
\end{equation}
and when $\cL = \partial_t^2 - \partial_x^2$, 
\begin{equation}\label{cond swe}
     G(t,x) = \frac{1}{2}\ind_{\{y\in \R : |y| < t\}}(x),  \quad \cI(t) = (G(t)\ast c_{W2}) + \partial_t(G(t)\ast c_{W1}) = tc_{W2} + c_{W1}. 
\end{equation}
A solution $u$ to \eqref{SPDE} is said to be unique if, for any other solution $\{v(t,x)\}_{(t,x) \in [0,T] \times \R}$, we have $P(u(t,x) = v(t,x)) = 1$ for every $(t,x) \in [0,T] \times \R$.
The following is a standard result (see \textit{e.g.} \cite{Dal99, DQ}). 
\begin{Prp}\label{Prp sol. wp}
Let $\sigma \colon \R \to \R$ be a Lipschitz function.  
The equation \eqref{SPDE} has a unique solution $\{u(t,x)\}_{(t,x) \in [0,T] \times \R}$, 
and the following holds.
\begin{enumerate}[\normalfont(i)]
    \item For all $p \geq 1$, $\sup_{(t,x) \in [0,T] \times \R}\norm{u(t,x)}_{L^p(\Omega)} < \infty$ and $(t,x) \mapsto u(t,x)$ is $L^p(\Omega)$-continuous.
    \item For each $t\geq 0$, the process $\{u(t,x)\}_{x\in \R}$ is strictly stationary.
\end{enumerate}
\end{Prp}

Moreover, it is known (\textit{cf.} \cite[Lemma 18]{Dal99}) that the solution $u$ to \eqref{SPDE} satisfies the property (S). 
For the definition of property (S), see \cite[Definition 5.1]{Dal99}. 
One of the consequences of property (S) is that for any $n \in \N$ and any $(t_1, x_1), \ldots, (t_n, x_n) \in [0,T] \times \R$, 
\begin{equation}
\label{stg stationarity}
    (u(t_1, x_1), \ldots, u(t_n, x_n)) \stackrel{d}{=} (u(t_1, x_1 + a), \ldots, u(t_n, x_n + a))
\end{equation}
holds for each $a \in \R$. 
Here, $\stackrel{d}{=}$ denotes equality in distribution between the left-hand side and the right-hand side.

\begin{Rem}
Strictly speaking, \cite[Lemma 18]{Dal99} shows that a sequence of approximations $u_n$ converging in $L^2(\Omega)$ to a solution $u$ of the SPDE considered there satisfies property (S) for each $n$. However, it is easy to see that $u$ itself also satisfies property (S), as a consequence of the $L^2(\Omega)$-convergence.
It is also worth noting that although a zero initial condition is considered in \cite{Dal99}, if the contribution from the initial data is deterministic and independent of the spatial variable, then property (S) can still be shown to hold by an argument similar to that in \cite[Lemma 4.2]{Ebi25}. 
\end{Rem}

For simplification, let us introduce the following notation: For $a,b \in \R$, $t\in [0,T]$, and $\mathbb{T}_k = (t_1, \ldots, t_k) \in [0,T]^k$,
\begin{align*}
    &F_{b}^a(t) \coloneqq \int_{a}^b (u(t,x) - \E[u(t,x)])dx, \quad \mathbb{F}^a_b(\mathbb{T}_k) \coloneqq (F_{b}^a(t_1), \ldots, F_{b}^a(t_k)), \quad \text{and} \quad \bF_b(\bT_k) \coloneqq \bF_b^0(\bT_k). 
\end{align*}

Using property (S) and the $L^2(\Omega)$-continuity of the solution $u$, we obtain the following invariance for its spatial average.
\begin{Lem}\label{Lem shift invariance}
Let $\{u(t,x)\}_{(t,x) \in [0,T] \times \R}$ be the solution to \eqref{SPDE}.
For every $k \in \N$, $\bT_k = (t_1, \ldots, t_k) \in [0,T]^k$, and $a, b \in \R$ with $a < b$, we have
\begin{equation}
    \bF_{b}^{a}(\bT_k) \stackrel{d}{=} \bF_{b+c}^{a+c}(\bT_k)
\end{equation}
for any $c \in \R$. In particular, the distribution of $\bF_{b+c}^{a+c}(\bT_k)$ does not depend on $c$. 
\end{Lem}

\begin{proof}
For each $i = 1, \ldots, k$, by Proposition \ref{Prp sol. wp}, $x \mapsto u(t_i, x)$ is uniformly $L^2(\Omega)$-continuous on compact sets of $\R$.  
Thus, for any $\varepsilon > 0$, there exists $\delta_{i, \varepsilon} > 0$ such that for any $x, x' \in \R$ with $|x - x'| < \delta_{i, \varepsilon}$, 
\begin{equation}
    \norm{u(t_i,x) - u(t_i, x')}_{L^2(\Omega)} < \varepsilon.
\end{equation}
For each $i$ and $\varepsilon$, let $a = r^{i, \varepsilon}_0 < r^{i, \varepsilon}_1 < \cdots < r^{i, \varepsilon}_{n_{i, \varepsilon}} = b$ be a partition of $[a,b]$ such that
\begin{equation}
    \max_{0 \leq j \leq n_{i, \varepsilon} - 1} |r^{i, \varepsilon}_j - r^{i, \varepsilon}_{j+1}| < \delta_{i, \varepsilon},
\end{equation}
and take $y_{j}^{i, \varepsilon} \in [r^{i, \varepsilon}_{j}, r^{i, \varepsilon}_{j+1}]$ arbitrarily for each $j$. 
Then, we can easily check that  
\begin{equation} 
    \left(\sum_{j=0}^{n_{1, \varepsilon} - 1} (u(t_1, y_j^{1, \varepsilon}) - \E[u(t_1, y_j^{1, \varepsilon})])\int_{r^{1, \varepsilon}_j}^{r^{1, \varepsilon}_{j+1}} dx, \ldots, \sum_{j=0}^{n_{k, \varepsilon} - 1} (u(t_k, y_j^{k, \varepsilon}) - \E[u(t_k, y_j^{k, \varepsilon})])\int_{r^{k, \varepsilon}_j}^{r^{k, \varepsilon}_{j+1}} dx \right) 
\end{equation}
converges in $L^2(\Omega; \R^k)$ to $\bF_b^a(\bT_k)$ as $\varepsilon \to 0$.
Now observe from \eqref{stg stationarity} that for any $c \in \R$, we have
\begin{align}
    &\left(\sum_{j=0}^{n_{1, \varepsilon} - 1} (u(t_1, y_j^{1, \varepsilon}) - \E[u(t_1, y_j^{1, \varepsilon})])\int_{r^{1, \varepsilon}_j}^{r^{1, \varepsilon}_{j+1}} dx, \ldots, \sum_{j=0}^{n_{k, \varepsilon} - 1} (u(t_k, y_j^{k, \varepsilon}) - \E[u(t_k, y_j^{k, \varepsilon})])\int_{r^{k, \varepsilon}_j}^{r^{k, \varepsilon}_{j+1}} dx \right) \\
    &\stackrel{d}{=} \left(\sum_{j=0}^{n_{1, \varepsilon} - 1} (u(t_1, y_j^{1, \varepsilon} - c) - \E[u(t_1, y_j^{1, \varepsilon} - c)] )\int_{r^{1, \varepsilon}_j}^{r^{1, \varepsilon}_{j+1}} dx, \ldots, \sum_{j=0}^{n_{k, \varepsilon} - 1} (u(t_k, y_j^{k, \varepsilon} - c) - \E[u(t_k, y_j^{k, \varepsilon} -c)])\int_{r^{k, \varepsilon}_j}^{r^{k, \varepsilon}_{j+1}} dx \right).
\end{align}
Since the right-hand side converges in $L^2(\Omega; \R^k)$ to $\bF_{b+c}^{a+c}(\bT_k)$ as $\varepsilon \to 0$, we can conclude that $F_{b}^{a}(\bT_k) \stackrel{d}{=} F_{b+c}^{a+c}(\bT_k)$, and the lemma follows.
\end{proof}

\subsection{Malliavin calculus}
Here we introduce some elements of the Malliavin calculus based on the Gaussian process $W = \{W(h) \mid h \in \mathcal{H}_T\}$ defined in Section \ref{subsection Solutions to stochastic heat and wave equations}.
For much more information, the reader is referred to \cite{Nualart_book}. 

We say that $F\colon \Omega \to \R$ is a smooth cylindrical random variable with respect to $W$ if $F$ is of the form $F=f(W(h_1),\ldots, W(h_k))$ for some $k \in \N$, $h_1, \ldots, h_k \in \mathcal{H}_T$, and a smooth function $f\colon \R^k \to \R$ such that $f$ and all orders of its partial derivatives are at most of the polynomial growth. 

We will write $D$ for the Malliavin derivative operator. 
If $F=f(W(h_1),\ldots, W(h_k))$ is a smooth cylindrical random variable, then the Malliavin derivative $DF$ is the $\mathcal{H}_T$-valued random variable given by
\begin{equation*}
    DF = \sum_{i=1}^k\frac{\partial f}{\partial x_i}(W(h_1),\ldots, W(h_k))h_i.
\end{equation*}
The operator $D$ is a closed operator from $L^p(\Omega)$ into $L^p(\Omega;\mathcal{H}_T)$ for any $p \in [1, \infty)$, and the domain of $D$ in $L^p(\Omega)$ is the Sobolev space $\mathbb{D}^{1,p}$ defined as the closure of the set of all smooth cylindrical random variables in $L^p(\Omega)$ under the norm 
\begin{equation*}
    \lVert F \rVert_{1,p} \coloneqq ( \mathbb{E}[|F|^p] + \mathbb{E}[\lVert DF \rVert_{\mathcal{H}_T}^p] )^{\frac{1}{p}}.
\end{equation*}

For any $F \in \mathbb{D}^{1,p}$ and a Lipschitz function $\phi \colon \R \to \R$, $\phi(F)$ belongs to $\mathbb{D}^{1,p}$ and we have the following chain rule: there exists a real-valued random variable $\Phi_F$ with $|\Phi_F| \leq \mathrm{Lip}(\phi)$ almost surely such that 
\begin{equation*}
    D \phi(F) = \Phi_{F} DF.
\end{equation*}
If $\phi$ is a continuously differentiable function with bounded derivative, we have $\Phi_{F} = \phi'(F)$.

Let $F,G \in \mathbb{D}^{1,2}$. 
If the derivatives $DF$ and $DG$ have good enough versions $D_{t,x}F$, $D_{t,x}G$ which are measurable functions on $[0,T] \times \R \times \Omega$, then we can derive from the Clark-Ocone formula that 
\begin{equation}\label{Covariance bound}
    |\mathrm{Cov}(F, G)|
    \leq \int_0^T \int_{\R^{2}}\lVert D_{s,y}F \rVert_{L^2(\Omega)}\lVert D_{s,z}G \rVert_{L^2(\Omega)}\Gamma(y-z)dydzds.
\end{equation}
We will use this to obtain the covariance estimates for the spatial averages of the solution in Section \ref{section Covariance estimates}.

\subsection{Tools from large deviation theory}
This section collects some basic notions and tools from the large deviation theory. 
See \cite{Dembo-Zeitouni} for more details.

Let $\cX$ be a topological space and $\cB(\cX)$ be its Borel $\sigma$-algebra. 
A function $I \colon \cX \to [0, \infty]$ is called a rate function if $I$ is lower semicontinuous. 
We say that a family of probability measures $\{\nu_R\}_{R>0}$ on $(\cX, \mathcal{B}(\cX))$ satisfies the LDP with speed $R$ and a rate function $I$ if, for all $A \in \mathcal{B}(\cX)$, 
\begin{equation}
    -\inf_{x \in A^{o}} I(x) \leq \liminf_{R \to \infty} \frac{1}{R}\log \nu_R(A) \leq \limsup_{R \to \infty}\frac{1}{R}\log \nu_R(A) \leq - \inf_{x \in \Bar{A}} I(x),
\end{equation}
where $A^{o}$ and $\Bar{A}$ denote the interior and the closure of $A$, respectively.
A rate function $I$ is said to be a good rate function if, for all $r \in [0,\infty)$, $I^{-1}([0,r])$ is compact.  
A family $\{\nu_R\}_{R>0}$ is called exponentially tight if for every $\alpha < \infty$, there exists a compact set $K_{\alpha} \subset \cX$ such that
\begin{equation}
    \limsup_{R \to \infty} \frac{1}{R}\log \nu_R(K_{\alpha}^{c}) < -\alpha.
\end{equation}
Let us recall the definition of a well-separating class given in \cite[Definition 4.4.7]{Dembo-Zeitouni}. 
A class $\cG$ of continuous real-valued functions on $\cX$ is called well-separating if it satisfies the following:
\begin{enumerate}
    \item[(i)] $\mathcal{G}$ contains the all  constant functions.
    \item[(ii)] If $g_1, g_2 \in \cG$, then $g_1 \land g_2 \in \cG$, where $g_1 \land g_2(x) \coloneqq \min{\{g_1(x), g_2(x)\}}$.
    \item[(iii)] For any $x,y \in \cX$ with $x \neq y$ and $a,b \in \R$, there exists $g \in \cG$ such that $g(x) = a$ and $g(y) = b$.
\end{enumerate}
For instance, when $\cX$ is a locally convex Hausdorff topological vector space, the class of all continuous, bounded above, concave functions on $\cX$ is well-separating, as shown in \cite[Lemma 4.4.8]{Dembo-Zeitouni}.

In Section \ref{section LDP for the finite-dimensional distributions}, we establish the LDP for $\{\mu_R^{\mathbb{T}_k}\}_{R>0}$ by using the following form of inverse Varadhan's lemma.
Let $C_{\mathrm{b}}(\R^k)$ be the space of continuous bounded functions on $\R^k$.

\begin{Prp}[{\cite[Theorems 4.4.2 and 4.4.10]{Dembo-Zeitouni}}]
\label{Prp Inverse Varadhan's lemma}
Let $\{\nu_R\}_{R>0}$ be a family of probability measures on $\R^k$. 
Let $\mathcal{G}(\R^k)$ be a well-separating class on $\R^k$.
If $\{\nu_R\}_{R>0}$ is exponentially tight and the limit
\begin{equation}
\label{log mom gen limit}
    \Lambda_g \coloneqq \lim_{R \to \infty} \frac{1}{R}\log \left(\int_{\R^k} e^{Rg(x)}\nu_R(dx)\right)
\end{equation}
exists for all $g \in \mathcal{G}(\R^k)$, then $\Lambda_f$ exists for each $f \in C_{\mathrm{b}}(\R^k)$, and $\{\nu_R\}_{R>0}$ satisfies the LDP with speed $R$ and the good rate function 
\begin{equation*}
    I(x) = \sup_{f\in C_{\mathrm{b}}(\R^k)}\{f(x) - \Lambda_f\}.
\end{equation*}

\end{Prp}

In this paper, we will use the specific well-separating class $\mathcal{G}_{\mathrm{BLC}}(\R^k)$ on $\R^k$. 
\begin{Lem}[\textit{cf.} {\cite[Lemma 4.4.8]{Dembo-Zeitouni}}]
\label{Lem G blc}
The class $\mathcal{G}_{\mathrm{BLC}}(\R^k)$ of all Lipschitz continuous, bounded above, concave functions on $\R^k$ is well-separating.
\end{Lem}
\begin{proof}
The same argument as in \cite[Lemma 4.4.8]{Dembo-Zeitouni} works since the pointwise minimum of two Lipschitz continuous functions is again a Lipschitz continuous function.
\end{proof}

If the rate function $I$ in Proposition \ref{Prp Inverse Varadhan's lemma} is convex, we can show under a mild assumption that $I$ is the Fenchel-Legendre transform of 
\begin{equation}
\label{Lambda def}
    \Lambda(\lambda) \coloneqq \lim_{R \to \infty}\frac{1}{R}\log \left(\int_{\R^k} e^{R  \lambda \cdot x}\nu_R(dx)\right).
\end{equation}

\begin{Prp}[{\cite[Theorem 4.5.10]{Dembo-Zeitouni}}]
\label{Prp convexity and duality}
Let $\{\nu_R\}_{R>0}$ be a family of probability measures on $\R^k$ such that
\begin{equation*}
    \limsup_{R \to \infty} \frac{1}{R}\log \left(\int_{\R^k} e^{R \lambda \cdot x}\nu_R(dx)\right)< \infty, \quad \text{for any $\lambda \in \R^k$}.
\end{equation*}
If $\{\nu_R\}_{R>0}$ satisfies the LDP with speed $R$ and a good rate function $I$, then for each $\lambda \in \R^k$, the limit \eqref{Lambda def} exists, is finite, and satisfies $\Lambda(\lambda) = \sup_{x \in \R^k}\{ \lambda \cdot x - I(x)\}$.
Moreover, if the rate function $I$ is convex, then it is the Fenchel-Legendre transform of $\Lambda$. 
That is, for all $x \in \R^k$, 
\begin{equation*}
    I(x) = \sup_{\lambda \in \R^k}\{ \lambda \cdot x - \Lambda(\lambda)\}.
\end{equation*}
\end{Prp}

In Section \ref{subsection Existence of the limit}, we will use the following result shown in \cite{Hammersley} to ensure the existence of a limit \eqref{log mom gen limit}. 

\begin{Lem}[Approximate subadditivity, \  {\cite[Theorem 2]{Hammersley}}] 
\label{Lem approximate subadditivity}
Let $\zeta \colon (0,\infty) \to \R$ be a non-decreasing function on $[\xi, \infty)$ for some $\xi > 0$  with $\dis \int_{\xi}^{\infty}\frac{\zeta(x)}{x^2}dx < \infty$.
If $f \colon (0,\infty) \to \R$ satisfies 
\begin{equation*}
    f(x+y) \leq f(x) + f(y) + \zeta(x+y), \quad x,y \geq c_f,
\end{equation*}
for some $c_f >0$ which may depend on $f$, then the limit $\dis \Bar{f} = \lim_{x \to \infty} \frac{f(x)}{x} \in [-\infty, \infty)$ exists.
\end{Lem}

We also record the large deviations result for projective limits by Dawson and G\"{a}rtner \cite{Dawson-Gartner} in a particular form that can be applied to prove Theorem \ref{Thm main 2}.
For the general result and more details, we refer to \cite[Section 4.6]{Dembo-Zeitouni}.

Let 
\begin{equation}
    \mathcal{J} \coloneqq \bigcup_{k \in \N} \{(t_1, \ldots, t_k) \in \R^{k} \mid 0 \leq t_1 < \cdots < t_k \leq T\}.
\end{equation}
That is, for each $j \in \mathcal{J}$, there exists $k\in \N$ such that $j=(t_i, \ldots, t_k)$ with $0 \leq t_1< \cdots < t_k \leq T$, and we write $|j|$ for this $k \in \N$ for simplicity. 
We define a partial order $\preceq$ on $\mathcal{J}$ as follows: For $i = (s_1, \ldots, s_{|i|}) \in \mathcal{J}$ and $j= (t_1, \ldots, t_{|j|}) \in \mathcal{J}$, $i \preceq j$ if and only if $\{s_1, \ldots, s_{|i|}\} \subset \{t_1, \ldots, t_{|j|}\}$.
When $i \preceq j$, there exists an injection $a \colon \{1, \ldots, |i|\} \to \{1, \ldots, |j|\}$ and we define the projection $p_{ij}$ by 
\begin{equation}
    p_{ij} \colon \R^{|j|} \to \R^{|i|}, \quad (x_1, \ldots, x_{|j|}) \mapsto (x_{a(1)}, \ldots, x_{a(|i|)}).
\end{equation}
Then, $(\mathcal{J}, \preceq)$ is a partially ordered right-filtering set and $((\R^{|j|})_{j \in \mathcal{J}}, (p_{ij})_{i \preceq j \in \mathcal{J}})$ is a projective system. 
The projective limit $\varprojlim \R^{|j|}$ of the system is given by
\begin{equation}
    \varprojlim \R^{|j|} \coloneqq \left\{(x_j)_{j \in \mathcal{J}} \in \prod_{j \in \mathcal{J}}\R^{|j|} \relmiddle{|} x_i = p_{ij}(x_j) \quad \text{for every $i,j \in \mathcal{J}$ with $i \preceq j$} \right\},
\end{equation}
and the topology on $\varprojlim \R^{|j|}$ is the topology induced by the product topology of $\prod_{j \in \mathcal{J}}\R^{|j|}$.
For each $j \in \mathcal{J}$, we write $p_j \colon \varprojlim \R^{|j|} \to \R^{|j|}$ for the canonical projection.

Let $\R^{[0,T]}$ be the product space and let $\widetilde{p}_j \colon \R^{[0,T]} \to \R^{|j|}$ be the map given by $\widetilde{p}_j(f) = (f(t_1), \ldots, f(t_{|j|}))$ for each $j = (t_1, \ldots, t_{|j|}) \in \mathcal{J}$.
Then, $\Phi \colon \R^{[0,T]} \to \varprojlim \R^{|j|}$ defined by $\Phi(f) = (\widetilde{p}_{j}(f))_{j \in \mathcal{J}}$ gives a homeomorphism, and we can identify $\varprojlim \R^{|j|}$ with $\R^{[0,T]}$ and $p_j$ with $\widetilde{p}_j$. 
Moreover, a Borel measure $\nu$ on $\R^{[0,T]}$ can naturally be regarded as a Borel measure on $\varprojlim \R^{|j|}$.
This identification allows us to use the results of large deviations for Borel probability measures on the projective limit $\varprojlim \R^{|j|}$ as results for Borel probability measures on $\R^{[0,T]}$.

\begin{Prp}\label{Prp proj gene}
Let $\{\nu_R\}_{R>0}$ be a family of probability measures on $(\R^{[0,T]}, \mathcal{B}(\R^{[0,T]}))$.
If, for every $j \in \cJ$, the family of Borel probability measures $\{\nu_R \circ \widetilde{p}_j^{\: -1}\}_{R>0}$ on $\R^{|j|}$ satisfies the LDP with the good rate function $I_j$, then $\{\nu_R\}_{R>0}$ satisfies the LDP with the good rate function 
\begin{equation}
    I(f) = \sup_{j \in \cJ}\left\{I_j(\widetilde{p}_j(f))\right\}, \qquad f\in \R^{[0,T]}.
\end{equation}
\end{Prp}

\section{Covariance estimates}
\label{section Covariance estimates}
For the rest of the paper, we proceed under Assumption \ref{Assum 1}.
We also recall that the notational convention introduced at the end of Section \ref{section Introduction} applies throughout the paper. 
Unless otherwise specified, we treat the cases of the stochastic heat and wave equations simultaneously in what follows.

We use the Malliavin calculus tool to bound  covariance functions.
If $\sigma$ is a continuously differentiable Lipschitz function, it is known that the solution $u(t,x)$ to \eqref{SPDE} belongs to $\mathbb{D}^{1,p}$ for any $(t,x) \in [0,T]\times \R$ and $p \in [1,\infty)$.
Moreover, we can take a version $D_{s,y}u(t,x)$ of the derivative $Du(t,x)$ such that $D_{s,y}u(t,x)$ is 
jointly measurable in the variables $(t,x,s,y,\omega)$ and satisfies
\begin{equation}
\label{sol derivative}
    D_{s,y}u(t,x) = G(t-s, x-y)\sigma(u(s,y)) + \int_{s}^t\int_{\R}G(t-r, x-z)\sigma'(u(r,z))D_{s,y}u(r,z)W(dr,dz).
\end{equation}
We have the following moment estimate for $D_{s,y}u(t,x)$. 

\begin{Lem}\label{Lem derivative moment}
Let $u(t,x)$ be the solution to $\eqref{SPDE}$.
For any $(t,x) \in [0,T]\times \R$ and $p \in [1,\infty)$, we have
\begin{equation}
    \lVert D_{s,y}u(t,x) \rVert_{L^p(\Omega)} \leq C_{p,T,\sigma}G(t-s, x-y), \qquad \text{for almost every $(s,y) \in [0,T] \times \R$.}
\end{equation}
\end{Lem}
For the proof of Lemma \ref{Lem derivative moment}, see \cite[Lemma A.1]{1dSHECLT} for the stochastic heat equation case and see \cite[Lemma 2.2]{1dSWECLT} and \cite[Theorem 1.3]{SWE12d} for the stochastic wave equation case. 

\begin{Rem}
To be precise, in \cite{1dSHECLT, 1dSWECLT, SWE12d}, the estimate in Lemma \ref{Lem derivative moment} is derived only under the setting where $\cI(t) = 1$. 
In our case, $\cI(t)$ may depend on time, but as long as it is deterministic, its Malliavin derivative vanishes, and we still have \eqref{sol derivative}, just as in the case $\cI(t) = 1$.
Therefore, Lemma \ref{Lem derivative moment} can be proved by the same argument as in the cited works, where the case $\cI(t) = 1$ is considered.
\end{Rem}

Recall from Section \ref{subsection Solutions to stochastic heat and wave equations} the following notation: For $a,b \geq 0$, $t\in [0,T]$, and $\mathbb{T}_k = (t_1, \ldots, t_k) \in [0,T]^k$,
\begin{align*}
    &F_{b}^a(t) \coloneqq \int_{a}^b (u(t,x) - \E[u(t,x)])dx, \quad \mathbb{F}^a_b(\mathbb{T}_k) \coloneqq (F_{b}^a(t_1), \ldots, F_{b}^a(t_k)), \quad \text{and} \quad \bF_b(\bT_k) \coloneqq \bF_b^0(\bT_k). 
\end{align*}
\begin{Lem}
\label{Lem spatial average derivative moment bounds}
For any $a, b \in \R$, $t \in [0,T]$, and $p\in [1,\infty)$, we have
\begin{equation*}
    \left\lVert D_{s,y}F_b^a(t) \right\rVert_{L^p(\Omega)} \leq C_{p,T,\sigma} (\ind_{[a,b]} \ast G(t-s))(y), \qquad \text{for almost every $(s,y) \in [0,T] \times \R$.}
\end{equation*} 
\end{Lem}
\begin{proof}
$F_b^a(t)$ can be viewed as a Bochner integral valued in $L^p(\Omega)$, and, from the closedness of the differential operator $D$, it follows that $DF_b^a(t) = \int_a^b Du(t,x)dx$. 
It is easily seen that $\int_a^b D_{s,y}u(t,x)dx$ is a version of $DF_b^a(t)$ (\textit{i.e.} $DF_b^a(t) = \int_a^b D_{\cdot,\star}u(t,x)dx$ in $L^p(\Omega; \cH_T)$), and Lemma \ref{Lem spatial average derivative moment bounds} follows from Lemma \ref{Lem derivative moment}.
\end{proof}

By combining \eqref{Covariance bound} and Lemma \ref{Lem spatial average derivative moment bounds}, we obtain the following covariance estimate.
\begin{Prp}
\label{Prp Cov estimate}
Let $L, R >0$, $k \in \N$, and $\varphi$, $\psi$ be Lipschitz functions on $\R^k$. 
For any $\mathbb{T}_k = (t_1, \ldots, t_k) \in [0,T]^k$ and for sufficiently large $\Theta > 0$, we have
\begin{equation*}
    \left|\mathrm{Cov}\left[\varphi(\bF_L(\mathbb{T}_k)),  \psi(\bF_{L+ \Theta +R}^{L+ \Theta}(\bT_k))\right]\right|
    \leq C_{T} \mathrm{Lip}(\varphi) \mathrm{Lip}(\psi) k^2LR \exp\{-C_{T, \eta, \rho}\Theta^{2 \land \eta}\}.
\end{equation*}

\end{Prp}

\begin{proof}
We first consider the stochastic heat equation case (\textit{i.e.} $u(t,x)$ satisfies \eqref{sol eq} with \eqref{cond she}).
By the chain rule, $\varphi(\bF_L(\mathbb{T}_k))$ and $\psi(\bF_{L+ \Theta +R}^{L+ \Theta}(\bT_k))$ belong to $\bD^{1,2}$ and 
\begin{equation}
    D\varphi(\bF_L(\mathbb{T}_k)) = \sum_{i=1}^{k}\Phi_i DF_L(t_i), \qquad D\psi(\bF_{L+ \Theta +R}^{L+ \Theta}(\bT_k)) = \sum_{i=1}^{k}\Psi_i DF_{L+ \Theta +R}^{L+ \Theta}(t_i),
\end{equation}
where $\Phi_i, \Psi_i$ are random variables bounded by $\mathrm{Lip}(\varphi), \mathrm{Lip}(\psi)$, respectively.
Applying \eqref{Covariance bound} and Lemma \ref{Lem spatial average derivative moment bounds}, we get
\begin{align}
    &\left|\mathrm{Cov}\left[\varphi(\bF_L(\mathbb{T}_k)),  \psi(\bF_{L+ \Theta +R}^{L+ \Theta}(\bT_k))\right]\right|\\
    &\leq C_{T, \sigma} \mathrm{Lip}(\varphi) \mathrm{Lip}(\psi) \sum_{i,j = 1}^{k}\int_0^{t_i \land t_j}\int_{\R^2}\left(\ind_{[0,L]}\ast G(t_i-s)\right)(y)\left(\ind_{[L+\Theta,L+\Theta+R]}\ast G(t_j-s)\right)(z)\Gamma(y-z)dydzds.
\end{align}
By a simple computation, 
\begin{align}
    &\int_0^{t_i \land t_j}\int_{\R^2}\left(\ind_{[0,L]}\ast G(t_i-s)\right)(y)\left(\ind_{[L+\Theta,L+\Theta+R]}\ast G(t_j-s)\right)(z)\Gamma(y-z)dydzds\\
    &= 
    \int_0^{t_i \land t_j}ds\int_0^L dw \int_{L+\Theta}^{L+\Theta+R}d\widetilde{w} (G(t_i-s)\ast G(t_j-s)\ast \Gamma)(w-\widetilde{w})\\
    &\leq C_T 
    \int_0^L dw \int_{L+\Theta}^{L+\Theta+R}d\widetilde{w}(\widetilde{G}\ast \Gamma)(w-\widetilde{w}),
\end{align}
where $\widetilde{G}(x) = e^{-\frac{|x|^2}{2T}}$.
Since $(\widetilde{G}\ast \Gamma)(x) = O\left(\exp\left\{-C_{T, \eta, \rho}|x|^{2 \land \eta} \right\}\right)$ as $|x| \to \infty$ by Lemma \ref{Lem exp decay convolution} below, the desired estimate follows. 
The same estimate also holds for the stochastic wave equation case since $\frac{1}{2}\ind_{\{|y|<t\}}(x) \leq C_T \frac{1}{\sqrt{2\pi t}}e^{-\frac{|x|^2}{2t}}$ for $t \in (0,T]$ and $x \in \R$.
This completes the proof.
\end{proof}

\begin{Lem}
\label{Lem exp decay convolution}
Let $f_i \colon \R \to \R \quad (i = 1,2)$ be locally integrable functions such that for some constants $K_i > 1$, $C_i>0$, $\alpha_i>0$, and $\beta_i > 0$,
\begin{equation}
    |f_i(x)| \leq C_i \exp\left(-\alpha_i|x|^{\beta_i}\right), \qquad |x| > K_i, \quad i =1,2. 
\end{equation}
Then $f_1, f_2 \in L^1(\R)$ and for any $x$ with $|x| > 2(K_1 \lor K_2)$, we have
\begin{equation}
    |(f_1 \ast f_2)(x)| \leq 2(C_1\norm{f_2}_{L^1(\R)} \lor C_2\norm{f_1}_{L^1(\R)})\exp\left(- \left(\frac{\alpha_1}{2^{\beta_1}} \land \frac{\alpha_2}{2^{\beta_2}} \right)|x|^{\beta_1 \land \beta_2}\right).
\end{equation}
\end{Lem}

\begin{proof}
Obviously, $f_1, f_2 \in L^1(\R)$. 
If $|x| > 2(K_1 \lor K_2)$ and $|y| \leq \frac{|x|}{2}$, then $|x-y| \geq |x| - |y| \geq \frac{|x|}{2} > (K_1 \lor K_2)$, and we have 
\begin{align*}
    |(f_1 \ast f_2)(x)| 
    &\leq 
    \int_{\{|y| \leq \frac{|x|}{2}\}}|f_1(x-y)||f_2(y)|dy + \int_{\{|y| > \frac{|x|}{2}\}}|f_1(x-y)||f_2(y)|dy\\
    &\leq 
    C_1\int_{\{|y| \leq \frac{|x|}{2}\}}\exp (-\alpha_1|x-y|^{\beta_1})|f_2(y)|dy + C_2\int_{\{|y| > \frac{|x|}{2}\}}|f_1(x-y)|\exp (-\alpha_2|y|^{\beta_2})dy\\
    &\leq 
    C_1\int_{\{|y| \leq \frac{|x|}{2}\}}|f_2(y)|dy \exp \left(-\frac{\alpha_1}{2^{\beta_1}}|x|^{\beta_1}\right) + C_2\int_{\{|y| > \frac{|x|}{2}\}}|f_1(x-y)|dy\exp \left(-\frac{\alpha_2}{2^{\beta_2}}|x|^{\beta_2}\right)\\
    &\leq 
    C_1 \lVert f_2 \rVert_{L^1(\R)}\exp \left(-\frac{\alpha_1}{2^{\beta_1}}|x|^{\beta_1}\right) + C_2\lVert f_1 \rVert_{L^1(\R)} \exp \left(-\frac{\alpha_2}{2^{\beta_2}}|x|^{\beta_2}\right).
\end{align*}
The desired inequality now follows from a straightforward estimate. 
\end{proof}

\section{LDP for the finite-dimensional distributions}
\label{section LDP for the finite-dimensional distributions}
We prove Theorem \ref{Thm main 1} in this section. 
The flow of the proof is essentially the same as in \cite[Section 6.4]{Dembo-Zeitouni}, but in our setting, the boundedness of random variables assumed in \cite{Dembo-Zeitouni} is not satisfied.
Instead, we proceed by using tail estimates of spatial integrals.

Section \ref{subsection Exponential tightness} and Section \ref{subsection Existence of the limit} are devoted to proving the exponential tightness of $\{\mu_R^{\mathbb{T}_k}\}_{R > 0}$ and the existence of the limit $\lim_{R \to \infty}R^{-1}\log \E[e^{Rg(\frac{1}{R}\mathbb{F}_R(\mathbb{T}_k))}]$ for each $g \in \mathcal{G}_{\mathrm{BLC}}(\R^k)$, respectively.  
It follows that $\{\mu_R^{\mathbb{T}_k}\}_{R > 0}$ satisfies the LDP with speed $R$ and the good rate function $I^{\mathbb{T}_k}$, thanks to Proposition \ref{Prp Inverse Varadhan's lemma}. 
We then apply Proposition \ref{Prp convexity and duality} to show that $I^{\mathbb{T}_k}$ is convex and can be expressed as the Fenchel-Legendre transform in Section \ref{subsection Convexity of the rate function}.

\subsection{Exponential tightness}
\label{subsection Exponential tightness}
\begin{Prp}
\label{Prp exponential tightness}
For every $k \in \N$ and $\mathbb{T}_k = (t_1, \ldots, t_k) \in [0,T]^k$, the family of probability measures $\{\mu_R^{\mathbb{T}_k}\}_{R > 0}$ is exponentially tight.
\end{Prp}

\begin{proof}
If we define
\begin{equation}
    F_{R,t}(r) = \int_{0}^r\int_{\R}(G(t-s) \ast \ind_{[0,R]})(y)\sigma(u(s,y))W(ds,dy),
\end{equation}
then $\{F_{R,t}(r)\}_{r \in [0,T]}$ is a square-integrable continuous martingale such that $F_{R,t}(t) = F_R(t)$ and 
\begin{align}
    \langle F_{R,t} \rangle_t 
    &= \int_0^{t}\int_{\R}(G(t-s)\ast \ind_{[0,R]})(y)(G(t-s)\ast \ind_{[0,R]})(z)\sigma(u(s,y))\sigma(u(s,z))\Gamma(y-z) dydzds \\
    &\leq C_T\norm{\sigma}_{\infty}^2\norm{\Gamma}_{L^1(\R)}R \eqqcolon C_{T, \Gamma, \sigma}R 
\end{align}
Let $\alpha >0$. 
The exponential inequality for martingales (\textit{cf.} \cite[Section A.2]{Nualart_book}) yields 
\begin{align*}
    \mu_R^{\mathbb{T}_k}\left(\left(\left[-\sqrt{2\alpha C_{T, \Gamma, \sigma}},  \sqrt{2\alpha C_{T, \Gamma, \sigma}} \right]^k\right)^c\right) 
    &\leq \sum_{i=1}^{k}P\left(\langle F_{R,t_i} \rangle_{t_i} \leq C_{T, \Gamma, \sigma}R, \  |F_{R,t_i}(t_i)| > \sqrt{2\alpha C_{T, \Gamma, \sigma}} R \right) \\
    &\leq 2k e^{-\alpha R},
\end{align*}
which implies the exponential tightness.
\end{proof}

\subsection{Existence of the limit}
\label{subsection Existence of the limit} 
In this section, we show that the limit 
\begin{equation}
    \Lambda_g^{\mathbb{T}_k} \coloneqq \lim_{R \to \infty}\frac{1}{R}\log 
    \E\left[\exp{\left\{R g\left(\frac{1}{R}\mathbb{F}_R(\mathbb{T}_k)\right)\right\}}\right]
\end{equation}
exists for any $k \in \N$, $\bT_k = (t_1, \ldots, t_k) \in [0, T]^k$, and $g \in \cG_{\mathrm{BLC}}(\R^k)$, by using approximate subadditivity lemma (Lemma \ref{Lem approximate subadditivity}). 
To this end, we first observe that it suffices to prove the existence of the limit $\Lambda_g^{\bT_k}$ for any $k \in \N$, $\bT_k \in [0,T]^k$, and $g \in \cG_{\mathrm{BLC}}(\R^k)$ satisfying $g(x) < 0$ for all $x \in \R^k$. 
Indeed, if this is established, then for any $h \in \cG_{\mathrm{BLC}}(\R^k)$ satisfying $h(x) < M_h$ for all $x \in \R^k$ for some constant $M_h > 0$, the function $h - M_h$ belongs to $\cG_{\mathrm{BLC}}(\R^k)$ and the limit 
\begin{equation}
    \Lambda_{h}^{\bT_k} = M_h + \Lambda_{h - M_h}^{\bT_k}
\end{equation}
also exists.

In what follows, we fix $\bT_k \in [0,T]^k$ and a function $g \in \cG_{\mathrm{BLC}}(\R^k)$ with $g(x) < 0$ for all $x \in \R^k$, and prove the existence of the limit $\Lambda_{g}^{\bT_k}$.
Let us write
\begin{equation}
    H^{\mathbb{T}_k}_R(g) \coloneqq  \E\left[\exp{\left\{R g\left(\frac{1}{R}\mathbb{F}_R(\mathbb{T}_k)\right)\right\}}\right]. 
\end{equation}

\begin{Lem}
\label{Lem les 5}
For sufficiently large $R>0$, 
\begin{equation*}
    H^{\mathbb{T}_k}_{R}(g) \geq \frac{1}{2}e^{-m_gR}, \qquad \text{where $m_g \coloneqq - \min_{x\in [-1,1]^k}g(x) \in (0,\infty)$}.
\end{equation*}
\end{Lem}
\begin{proof}
By a simple estimate, 
\begin{align*}
    H^{\mathbb{T}_k}_R(g) 
    \geq e^{-m_gR}P\left( \frac{1}{R}\mathbb{F}_{R}(\mathbb{T}_k) \in [-1,1]^k \right)
    \geq e^{-m_gR}\left(1- \sum_{i=1}^k P\left( |F_{R}(t_i)| > R \right) \right).
\end{align*}
Since $\lim_{R \to \infty}P\left( |F_{R}(t_i)| > R \right) = 0$ by the same argument as in the proof of Proposition \ref{Prp exponential tightness}, we have 
\begin{equation*}
    1- \sum_{i=1}^k P\left( |F_{R}(t_i)| > R \right)  > \frac{1}{2}, \qquad \text{for sufficiently large $R$,}
\end{equation*}
and the lemma follows.
\end{proof}

To show that $R \mapsto -\log H^{\mathbb{T}_k}_R(g)$ is approximately subadditive, we begin by evaluating $H^{\mathbb{T}_k}_{L+R}(g)$.

\begin{Lem}
\label{Lem application of reverse HI}
Let $\Theta, L, R >0$ and $p \in (0,1)$.
It holds
\begin{equation}
    H^{\mathbb{T}_k}_{L+R}(g) 
    \geq 
    \frac{\E\left[\exp{\left\{Lg\left(\frac{1}{L}\mathbb{F}_{L}(\mathbb{T}_k)\right) + Rg\left( \frac{1}{R}\mathbb{F}_{L+\Theta + R}^{L+\Theta} (\mathbb{T}_k) \right)  \right\}} \right]^{\frac{1}{p}}}{
    \E\left[\exp{\left\{ \frac{p}{1-p}\cdot \mathrm{Lip}(g) \left|\mathbb{F}_{L+\Theta}^{L}(\mathbb{T}_k) - \mathbb{F}_{L+R+\Theta}^{L+R}(\mathbb{T}_k) \right| \right\}} \right]^{\frac{1-p}{p}}}. \label{l bound 1}
\end{equation}
\end{Lem}

\begin{proof}
By the Lipschitz continuity of $g$, 
\small
\begin{align*}
    &\left| g\left(\frac{1}{L+R}\mathbb{F}_{L+R}(\mathbb{T}_k)\right) - g\left(\frac{L}{L+R}\cdot \frac{1}{L}\mathbb{F}_L(\mathbb{T}_k)  
    + \frac{R}{L+R}\cdot \frac{1}{R}\mathbb{F}^{L+\Theta}_{L+\Theta + R}(\mathbb{T}_k)\right) \right|
    \leq \frac{\mathrm{Lip}(g)}{L+R}\left| \mathbb{F}_{L+\Theta}^L(\mathbb{T}_k) - \mathbb{F}_{L+R+\Theta}^{L+ R}(\mathbb{T}_k) \right|.
\end{align*}
\normalsize
Hence, the concavity of $g$ yields
\small
\begin{align}
    H^{\mathbb{T}_k}_{L+R}(g)
    \geq 
    \E\left[\exp{\left\{Lg\left(\frac{1}{L}\mathbb{F}_{L}(\mathbb{T}_k)\right) + Rg\left( \frac{1}{R}\mathbb{F}_{L+\Theta + R}^{L+\Theta} (\mathbb{T}_k) \right)  \right\}} \exp{\left\{ - \mathrm{Lip}(g) \left|\mathbb{F}_{L+\Theta}^{L}(\mathbb{T}_k) - \mathbb{F}_{L+R+\Theta}^{L+R}(\mathbb{T}_k) \right| \right\}} \right] \label{r1}.
\end{align}
\normalsize
Applying a reverse H\"{o}lder inequality (\textit{cf.} \cite[THEOREM 2.12]{AdamsFournier}) to the right-hand side of \eqref{r1}, we obtain 
\begin{align*}
    H^{\mathbb{T}_k}_{L+R}(g)
    \geq 
    \frac{\E\left[\exp{\left\{L g\left(\frac{1}{L}\mathbb{F}_{L}(\mathbb{T}_k)\right) + R g\left( \frac{1}{R}\mathbb{F}_{L+\Theta + R}^{L+\Theta} (\mathbb{T}_k) \right)  \right\}} \right]^{\frac{1}{p}}}{
    \E\left[\exp{\left\{ \frac{p}{1-p}\cdot \mathrm{Lip}(g)  \left|\mathbb{F}_{L+\Theta}^{L}(\mathbb{T}_k) - \mathbb{F}_{L+R+\Theta}^{L+R}(\mathbb{T}_k) \right| \right\}} \right]^{\frac{1-p}{p}}},
\end{align*}
where we use the fact that $g(x) < 0$ for all $x \in \R^k$. 
\end{proof}

From now on, we fix $\alpha \in (\frac{1}{2 \land \eta}, 1)$ and $\beta > 0$ satisfying $\alpha + \beta <1$. 
Choosing 
\begin{equation*}
    \Theta = (L+R)^{\alpha} \quad 
    \text{and} \quad p = \frac{(L+R)^\beta}{(L+R)^\beta + 1} \in (0,1)
\end{equation*}
in Lemma \ref{Lem application of reverse HI}, we further estimate from below the right-hand side of \eqref{l bound 1}, \textit{i.e.}
\begin{equation}
\label{q1}
    \frac{\E\left[\exp{\left\{Lg\left(\frac{1}{L}\mathbb{F}_{L}(\mathbb{T}_k)\right) + Rg\left( \frac{1}{R}\mathbb{F}_{L+(L+R)^{\alpha} + R}^{L+(L+R)^{\alpha}} (\mathbb{T}_k) \right)  \right\}} \right]^{1 + \frac{1}{(L+R)^{\beta}}}}{
    \E\left[\exp{\left\{ \mathrm{Lip}(g) (L+R)^{\beta} \left|\mathbb{F}_{L+(L+R)^{\alpha}}^{L}(\mathbb{T}_k) - \mathbb{F}_{L+R+(L+R)^{\alpha}}^{L+R}(\mathbb{T}_k) \right| \right\}} \right]^{\frac{1}{(L+R)^{\beta}}}}.
\end{equation}
We first bound from above the denominator of \eqref{q1}.
\begin{Lem}
\label{Lem denominator estimate}
Under the same setting as in Lemma \ref{Lem application of reverse HI}, 
\begin{align}
    &\E\left[\exp{\left\{ \mathrm{Lip}(g) (L+R)^{\beta} \left|\mathbb{F}_{L+(L+R)^{\alpha}}^{L}(\mathbb{T}_k) - \mathbb{F}_{L+R+(L+R)^{\alpha}}^{L+R}(\mathbb{T}_k) \right| \right\}} \right]^{\frac{1}{(L+R)^{\beta}}}\\
    &\leq \left(C_{\mathrm{Lip}(g), T, \Gamma, \sigma}(L+R)^{\frac{\alpha}{2}+\beta}\right)^{\frac{1}{(L+R)^{\beta}}}\exp{\left\{ C_{\mathrm{Lip}(g), T, \Gamma, \sigma}(L+R)^{\alpha + \beta} \right\}}.
\end{align}
\end{Lem}

\begin{proof}
Since $\E[e^{cX}] = c\int_{\R}e^{cr}P(X \geq r)dr$ holds for any random variable $X$ and for any constant $c>0$, we get 
\begin{align}
    &\E\left[\exp{\left\{ \mathrm{Lip}(g) (L+R)^{\beta} \left|\mathbb{F}_{L+(L+R)^{\alpha}}^{L}(\mathbb{T}_k) - \mathbb{F}_{L+R+(L+R)^{\alpha}}^{L+R}(\mathbb{T}_k) \right| \right\}} \right] \\
    &\leq \mathrm{Lip}(g) (L+R)^{\beta}\sum_{i=1}^{k}\int_{\R}e^{\mathrm{Lip}(g) (L+R)^{\beta} r}P\left(\left|F_{L+(L+R)^{\alpha}}^{L}(t_i) - F_{L+R+(L+R)^{\alpha}}^{L+R}(t_i) \right| \geq \frac{r}{\sqrt{k}} \right)dr, \label{after fubini}
\end{align}
and we need to evaluate the probability inside the integral. 
If we let
\begin{equation*}
    M_t(r) \coloneqq \int_0^{r \land t} \int_{\R}\left( \ind_{[L,L+(L+R)^{\alpha}]}\ast G(t-s)(y) - \ind_{[L+R, L+R+(L+R)^{\alpha}]}\ast G(t-s)(y) \right)\sigma(u(s,y))W(ds,dy),
\end{equation*}
then $\{M_{t}(r)\}_{r \in [0,T]}$ is a square-integrable continuous martingale such that $M_{t}(t) = F_{L+(L+R)^{\alpha}}^L(t) - F_{L+R+(L+R)^{\alpha}}^{L+R}(t)$ and
\begin{align}
    \langle M_{t}(\cdot) \rangle_t \leq C_T \norm{\sigma}_{\infty}^2\norm{\Gamma}_{L^1(\R)}(L+R)^{\alpha} \eqqcolon C_{T, \Gamma, \sigma}(L+R)^{\alpha}
\end{align}
as in the proof of Proposition \ref{Prp exponential tightness}.
Applying the exponential inequality for martingales, we obtain
\begin{equation}
    P\left(\left|F_{L+(L+R)^{\alpha}}^{L}(t) - F_{L+R+(L+R)^{\alpha}}^{L+R}(t) \right| \geq r \right) \leq 2 e^{-\frac{r^2}{2C_{T,\Gamma, \sigma} (L+R)^{\alpha}}}. \label{tail estimate 1}
\end{equation}
Hence, combining \eqref{after fubini} and \eqref{tail estimate 1}, we have
\begin{align*}
    &\E\left[\exp{\left\{ \mathrm{Lip}(g) (L+R)^{\beta}\left|\mathbb{F}_{L+(L+R)^{\alpha}}^{L}(\mathbb{T}_k) - \mathbb{F}_{L+R+(L+R)^{\alpha}}^{L+R}(\mathbb{T}_k) \right| \right\}} \right] \\
    &\leq \sqrt{8\pi C_{T, \Gamma, \sigma}} k^{\frac{3}{2}} \mathrm{Lip}(g)(L+R)^{\frac{\alpha}{2} + \beta}\exp{\left\{\frac{ k \mathrm{Lip}(g)^2C_{T, \Gamma, \sigma}}{2} (L+R)^{\alpha + 2\beta}\right\}},
\end{align*}
and the lemma follows.
\end{proof}

Next, we bound from below the numerator of \eqref{q1} using the covariance estimate in Proposition \ref{Prp Cov estimate}.

\begin{Lem}
\label{Lem numerator estimate}
For sufficiently large $L$ and $R$, 
\begin{align*}
    &\E\left[\exp{\left\{Lg\left(\frac{1}{L}\mathbb{F}_{L}(\mathbb{T}_k)\right) + Rg\left( \frac{1}{R}\mathbb{F}_{L+(L+R)^{\alpha} + R}^{L+(L+R)^{\alpha}} (\mathbb{T}_k) \right)  \right\}} \right]^{1 + \frac{1}{(L+R)^{\beta}}} 
    \geq
    \frac{1}{16}H^{\mathbb{T}_k}_L(g) H^{\mathbb{T}_k}_R(g) e^{-m_g(L+R)^{1-\beta}}
\end{align*}
\end{Lem}

\begin{proof}
By Proposition \ref{Prp Cov estimate}, Lemmas \ref{Lem les 5} and \ref{Lem shift invariance}, we have for sufficiently large $L$ and $R$, 
\begin{align*}
    &\E\left[\exp{\left\{Lg\left(\frac{1}{L}\mathbb{F}_{L}(\mathbb{T}_k)\right) + Rg\left( \frac{1}{R}\mathbb{F}_{L+(L+R)^{\alpha} + R}^{L+(L+R)^{\alpha}} (\mathbb{T}_k) \right)  \right\}} \right] \\
    &\geq 
    \E\left[\exp{\left\{L g\left(\frac{1}{L}\mathbb{F}_L(\mathbb{T}_k)\right)\right\}}\right]\E\left[\exp{\left\{Rg\left( \frac{1}{R}\mathbb{F}_{L+(L+R)^{\alpha} + R}^{L+(L+R)^{\alpha}} (\mathbb{T}_k) \right)\right\}}\right] - C_{\mathrm{Lip}(g), T}k^2LRe^{-C_{T, \eta, \rho}(L+R)^{(2 \land \eta)\alpha}}\\
    &\geq H^{\mathbb{T}_k}_L(g) H^{\mathbb{T}_k}_R(g) \left(1- 4C_{\mathrm{Lip}(g), T}k^2(L+R)^2 e^{-C_{T, \eta, \rho}(L+R)^{(2 \land \eta)\alpha} + m_g(L+R)}\right).
\end{align*}
Because $(2\land \eta)\alpha > 1$ and $\displaystyle \lim_{x \to \infty}x^2 e^{-C_{T, \eta, \rho}x^{(2\land \eta)\alpha} + m_g x} = 0$,
we can assume 
\begin{equation*}
    \frac{1}{2} < 1- 4C_{\mathrm{Lip}(g), T}k^2(L+R)^2 e^{-C_{T, \eta, \rho}(L+R)^{(2\land \eta)\alpha} + m(L+R)} < 1
\end{equation*}
by taking $L$ and $R$ large enough.
Thus, we can apply Lemma \ref{Lem les 5} again to obtain
\begin{align*}
    &\E\left[\exp{\left\{Lg\left(\frac{1}{L}\mathbb{F}_{L}(\mathbb{T}_k)\right) + Rg\left( \frac{1}{R}\mathbb{F}_{L+(L+R)^{\alpha} + R}^{L+(L+R)^{\alpha}} (\mathbb{T}_k) \right)  \right\}} \right]^{1 + \frac{1}{(L+R)^{\beta}}} \\
    &\geq 
    \left(\frac{1}{2}\right)^{1 + \frac{1}{(L+R)^{\beta}}}H_L^{\mathbb{T}_k}(g) H_R^{\mathbb{T}_k}(g) \left(\frac{1}{4}e^{-m_g(L+R)}\right)^{\frac{1}{(L+R)^{\beta}}}\\
    &\geq \frac{1}{16}H_L^{\mathbb{T}_k}(g) H_R^{\mathbb{T}_k}(g) e^{-m_g(L+R)^{1-\beta}},
\end{align*}
which is the desired estimate.
\end{proof}

Now, we combine all the computations to show the approximate subadditivity of $R \mapsto -\log H_R^{\mathbb{T}_k}(g)$. 

\begin{Prp}
\label{Prp existence of the limit}
There exists $M > 0$ such that for any $L, R >M$, 
\begin{equation*}
    -\log H_{L+R}^{\mathbb{T}_k}(g) \leq - \log H_L^{\mathbb{T}_k}(g) - \log H_R^{\mathbb{T}_k}(g) + (m_g+1)(L+R)^{1-\beta} + C_{\mathrm{Lip}(g), T, \Gamma, \sigma}(L+R)^{\alpha + \beta}. 
\end{equation*}
Consequently, the limit $\dis \Lambda_g^{\mathbb{T}_k} = \lim_{R \to \infty} \frac{1}{R}\log H_R^{\mathbb{T}_k}(g)$ exists.
\end{Prp}
\begin{proof}
Take $L$ and $R$ large enough to apply Lemma \ref{Lem numerator estimate}. 
By Lemmas \ref{Lem application of reverse HI}, \ref{Lem denominator estimate}, and \ref{Lem numerator estimate}, we have
\begin{align*}
    \frac{H_{L+R}^{\mathbb{T}_k}(g)}{H_L^{\mathbb{T}_k}(g) H_R^{\mathbb{T}_k}(g)} 
    \geq \frac{1}{16\left(C_{\mathrm{Lip}(g), T, \Gamma, \sigma}(L+R)^{\frac{\alpha}{2}+\beta}\right)^{\frac{1}{(L+R)^{\beta}}}}\exp\left\{-m_g(L+R)^{1-\beta}-C_{\mathrm{Lip}(g), T, \Gamma, \sigma}(L+R)^{\alpha + \beta}\right\}
\end{align*}
It follows by letting $L+R$ large enough that
\begin{align*}
    -\log H_{L+R}^{\mathbb{T}_k}(g) 
    &\leq - \log H_L^{\mathbb{T}_k}(g) - \log H_R^{\mathbb{T}_k}(g) + (m_g+1)(L+R)^{1-\beta} + C_{\mathrm{Lip}(g), T, \Gamma, \sigma}(L+R)^{\alpha + \beta}.
\end{align*}
Since $1-\beta < 1$ and $\alpha + \beta <1$, the existence of the limit follows from Lemma \ref{Lem approximate subadditivity}.
\end{proof}

\subsection{Convexity of the rate function}
\label{subsection Convexity of the rate function}
By Propositions \ref{Prp Inverse Varadhan's lemma}, \ref{Prp exponential tightness}, and \ref{Prp existence of the limit}, we have already shown that $\{\mu_R^{\mathbb{T}_k}\}_{R>0}$ satisfies the LDP with speed $R$ and the good rate function
\begin{equation}
\label{rate function of mu_R^(t)}
    I^{\mathbb{T}_k}(x) = \sup_{f\in C_{\mathrm{b}}(\R^k)}\{f(x) - \Lambda_f^{\mathbb{T}_k}\},
\end{equation}
where 
\begin{equation*}
    \Lambda_f^{\mathbb{T}_k} = \lim_{R \to \infty}\frac{1}{R}\log 
    \E\left[\exp{\left\{R f\left(\frac{1}{R}\mathbb{F}_R(\mathbb{T}_k)\right)\right\}}\right], \quad f \in C_{\mathrm{b}}(\R^k).
\end{equation*}
In this section, we show the convexity of the rate function $I^{\mathbb{T}_k}$ in Proposition \ref{Prp convexity rate function} below.
It then follows from Proposition \ref{Prp convexity and duality} that $I^{\mathbb{T}_k}$ is the Fenchel-Legendre transform of 
\begin{equation}
\label{lambda t limit}
    \Lambda^{\mathbb{T}_k}(\lambda) = \lim_{R \to \infty}\frac{1}{R}\log \E[e^{\lambda \cdot \mathbb{F}_R(\mathbb{T}_k)}], \quad \lambda \in \R^k,
\end{equation}
thereby completing the proof of Theorem \ref{Thm main 1}.
Hereafter, we use the notations
\begin{equation*}
    \mathbb{B}_k(x,r) \coloneqq \{y \in \R^k : |y-x| < r\} \quad \text{and} \quad \mathbb{B}_k(r) \coloneqq \mathbb{B}_k(0,r), \quad x \in \R^k, r>0.
\end{equation*}

\begin{Lem}
\label{Lem convexity inequality}
If $x_1, x_2 \in \R^k$ and $\delta > 0$ satisfy 
\begin{equation*}
    \liminf_{R \to \infty}\frac{1}{R}\log \mu_R^{\mathbb{T}_k}\left(\mathbb{B}_k\left(x_i, \frac{\delta}{4}\right)\right) > -\infty, \quad i =1,2, 
\end{equation*}
then there exists $R_{x_1, x_2, \delta} > 0$ such that for any $R \geq R_{x_1, x_2, \delta}$, 
\begin{equation*}
    \mu_{2R}^{\mathbb{T}_k}\left(\mathbb{B}_k\left(\frac{x_1 + x_2}{2}, \delta \right)\right) \geq \frac{1}{2}\mu_{R}^{\mathbb{T}_k}\left(\mathbb{B}_k\left(x_1, \frac{\delta}{4} \right)\right)\mu_{R}^{\mathbb{T}_k}\left(\mathbb{B}_k\left(x_2, \frac{\delta}{4} \right)\right).
\end{equation*}
\end{Lem}

\begin{proof}
Let $\Theta > 0$.
By a simple estimate, 
\begin{align*}
    &\mu_{2R}^{\mathbb{T}_k}\left(\mathbb{B}_k\left(\frac{x_1 + x_2}{2}, \delta \right)\right) 
    =
    P\left(\left| \frac{1}{R}\mathbb{F}_R(\mathbb{T}_k) - x_1 + \frac{1}{R}\mathbb{F}_{2R}^{R}(\mathbb{T}_k) - x_2 \right| <2\delta \right)\\
    &\geq P\left( \left| \frac{1}{R}\mathbb{F}_R(\mathbb{T}_k) - x_1 \right| < \frac{\delta}{2},  \left| \frac{1}{R}\mathbb{F}_{2R + \Theta}^{R+ \Theta}(\mathbb{T}_k) - x_2 \right| < \frac{\delta}{2}, \left|\frac{1}{R}\mathbb{F}_{2R}^{R}(\mathbb{T}_k) -  \frac{1}{R}\mathbb{F}_{2R + \Theta}^{R + \Theta}(\mathbb{T}_k) \right| < \delta \right)\\
    &=
    \E\left[ \ind_{\mathbb{B}_k\left(x_1,\frac{\delta}{2}\right)}\left(\frac{1}{R}\mathbb{F}_R(\mathbb{T}_k)\right) \ind_{\mathbb{B}_k\left(x_2,\frac{\delta}{2}\right)}\left(\frac{1}{R}\mathbb{F}_{2R+\Theta}^{R+\Theta}(\mathbb{T}_k)\right) \ind_{\mathbb{B}_k\left(\delta \right)}\left(\frac{1}{R}\mathbb{F}_{R + \Theta}^{R}(\mathbb{T}_k) - \frac{1}{R}\mathbb{F}_{2R + \Theta}^{2R}(\mathbb{T}_k)\right)\right].
\end{align*}
Since 
\begin{align*}
    &\left|\E\left[ \ind_{\mathbb{B}_k\left(x_1,\frac{\delta}{2}\right)}\left(\frac{1}{R}\mathbb{F}_R(\mathbb{T}_k)\right) \ind_{\mathbb{B}_k\left(x_2,\frac{\delta}{2}\right)}\left(\frac{1}{R}\mathbb{F}_{2R+\Theta}^{R+\Theta}(\mathbb{T}_k)\right) \ind_{\mathbb{B}_k\left(\delta \right)}\left(\frac{1}{R}\mathbb{F}_{R + \Theta}^{R}(\mathbb{T}_k) - \frac{1}{R}\mathbb{F}_{2R + \Theta}^{2R}(\mathbb{T}_k)\right)\right] \right.\\[1mm]
    &\quad \left.
    - \E\left[ \ind_{\mathbb{B}_k\left(x_1,\frac{\delta}{2}\right)}\left(\frac{1}{R}\mathbb{F}_R(\mathbb{T}_k)\right) \ind_{\mathbb{B}_k\left(x_2,\frac{\delta}{2}\right)}\left(\frac{1}{R}\mathbb{F}_{2R+\Theta}^{R+\Theta}(\mathbb{T}_k)\right) \right]\right|\\[1mm]
    &\leq 
    P\left(\left|\mathbb{F}_{R + \Theta}^{R}(\mathbb{T}_k) - \mathbb{F}_{2R+\Theta}^{2R}(\mathbb{T}_k)\right| \geq R\delta \right)\\
    &\leq 2ke^{-\frac{\delta^2 R^2}{2 k \Theta C_{T, \Gamma, \sigma} }}
\end{align*}
by a similar argument as in the proof of Lemma \ref{Lem denominator estimate} (see \eqref{tail estimate 1}), we have
\begin{equation}
\label{le1}
    \mu_{2R}^{\mathbb{T}_k}\left(\mathbb{B}_k\left(\frac{x_1 + x_2}{2}, \delta \right)\right)
    \geq \E\left[ \ind_{\mathbb{B}_k\left(x_1,\frac{\delta}{2}\right)}\left(\frac{1}{R}\mathbb{F}_R(\mathbb{T}_k)\right) \ind_{\mathbb{B}_k\left(x_2,\frac{\delta}{2}\right)}\left(\frac{1}{R}\mathbb{F}_{2R+\Theta}^{R+\Theta}(\mathbb{T}_k)\right) \right] - 2ke^{-\frac{\delta^2 R^2}{2 k \Theta C_{T, \Gamma, \sigma}}}.
\end{equation}
Let $f_i^{\delta} \colon \R^k \to [0,1], \ (i=1,2)$ be Lipschitz functions such that $\ind_{\mathbb{B}_k\left(x_i,\frac{\delta}{4}\right)} \leq f_i^{\delta} \leq \ind_{\mathbb{B}_k\left(x_i,\frac{\delta}{2}\right)}$.
As before, we deduce from Proposition \ref{Prp Cov estimate} that
\begin{equation*}
    \left| \mathrm{Cov}\left( f_1^{\delta}\left(\frac{1}{R}\mathbb{F}_{R}(\mathbb{T}_k)\right), f_2^{\delta}\left(\frac{1}{R}\mathbb{F}_{2R+\Theta}^{R+\Theta}(\mathbb{T}_k)\right) \right) \right| \leq C_T \mathrm{Lip}(f_1^{\delta}) \mathrm{Lip}(f_2^{\delta})k^2 \exp\{-C_{T, \eta, \rho}\Theta^{2 \land \eta}\}.
\end{equation*}
This together with \eqref{le1} and Lemma \ref{Lem shift invariance} yields 
\begin{equation}
    \mu_{2R}^{\mathbb{T}_k}\left(\mathbb{B}_k\left(\frac{x_1 + x_2}{2}, \delta \right)\right) 
    \geq 
    \mu_{R}^{\mathbb{T}_k}\left(\mathbb{B}_k\left(x_1, \frac{\delta}{4} \right)\right)\mu_{R}^{\mathbb{T}_k}\left(\mathbb{B}_k\left(x_2, \frac{\delta}{4} \right)\right) - \mathbf{A}(\Theta, R), \label{c1}
\end{equation}
where $\dis \mathbf{A}(\Theta, R) \coloneqq C_T \mathrm{Lip}(f_1^{\delta}) \mathrm{Lip}(f_2^{\delta}) k^2\exp\{-C_{T, \eta, \rho}\Theta^{2 \land \eta}\} + 2ke^{-\frac{\delta^2 R^2}{2 k \Theta C_{T, \Gamma, \sigma}}}$. 
Now observe that by the assumption of the lemma, there exists $M >0$ and $R^{(1)}_{x_1,x_2,\delta} > 0$ such that for any $R \geq R^{(1)}_{x_1,x_2,\delta}$, 
\begin{equation}
\label{c2}
    \mu_{R}^{\mathbb{T}_k}\left(\mathbb{B}_k\left(x_1, \frac{\delta}{4} \right)\right)\mu_{R}^{\mathbb{T}_k}\left(\mathbb{B}_k\left(x_2, \frac{\delta}{4} \right)\right) \geq e^{-MR}.
\end{equation}
Hence, combining \eqref{c1} and \eqref{c2}, we obtain for any $R \geq R^{(1)}_{x_1,x_2,\delta}$,
\begin{equation}
    \mu_{2R}^{\mathbb{T}_k}\left(\mathbb{B}_k\left(\frac{x_1 + x_2}{2}, \delta \right)\right) 
    \geq 
    \mu_{R}^{\mathbb{T}_k}\left(\mathbb{B}_k\left(x_1, \frac{\delta}{4} \right)\right)\mu_{R}^{\mathbb{T}_k}\left(\mathbb{B}_k\left(x_2, \frac{\delta}{4} \right)\right)
    \left\{ 1- \mathbf{A}(\Theta, R)e^{MR} \right\}.
\end{equation}
Set $\Theta = R^{\gamma}$ for some $\gamma \in \left(\frac{1}{2 \land \eta}, 1\right)$. Then we can choose $R^{(2)}_{x_1,x_2,\delta}>0$ such that for any $R \geq R^{(2)}_{x_1,x_2,\delta}$, 
\begin{equation}
    1- \mathbf{A}(\Theta, R)e^{MR} \geq \frac{1}{2}.
\end{equation}
Therefore, letting $R_{x_1,x_2,\delta} \coloneqq R^{(1)}_{x_1,x_2,\delta} \lor R^{(2)}_{x_1,x_2,\delta}$ completes the proof.
\end{proof}

\begin{Prp}
\label{Prp convexity rate function}
For each $\lambda \in \R^k$, the limit $\Lambda^{\mathbb{T}_k}(\lambda)$ in \eqref{lambda t limit} exists and is finite.
Moreover, the rate function $I^{\mathbb{T}_k}$ in \eqref{rate function of mu_R^(t)} is convex. 
Consequently, $I^{\mathbb{T}_k}$ is the Fenchel-Legendre transform of $\Lambda^{\mathbb{T}_k}$, \textit{i.e.} 
\begin{equation*}
    I^{\mathbb{T}_k}(x) = \sup_{\lambda \in \R^k} \{ \lambda \cdot x - \Lambda^{\mathbb{T}_k}(\lambda) \}.
\end{equation*}
\end{Prp}

\begin{proof}
Recall that $\{F_{R,t}(r)\}_{r \in [0,T]}$ defined in the proof of Proposition \ref{Prp exponential tightness} is a square-integrable continuous martingale such that $F_{R,t}(t) = F_R(t)$ and $\langle F_{R,t} \rangle_t  \leq C_{T, \Gamma, \sigma} R$.
A simple computation shows that for any $\lambda \in \R^k$, 
\begin{align}
    \limsup_{R \to \infty} \frac{1}{R}\log \E[e^{\lambda \cdot \bF_R(\bT_k)}] 
    &\leq \limsup_{R \to \infty}  \frac{1}{R}\log \left( \prod_{i=1}^{k} \E\left[e^{k \lambda_i F_{R, t_i}(t_i)}\right] \right)\\
    &\leq \sum_{i=1}^{k} \limsup_{R \to \infty} \frac{1}{R}\log \left( \E\left[e^{k \lambda_i F_{R, t_i}(t_i) - \frac{k^2 \lambda_i^2}{2}\langle F_{R,t_i} \rangle_{t_i}}\right] e^{\frac{k^2 \lambda_i^2 C_{T, \Gamma, \sigma}}{2} R} \right)\\
    &\leq \frac{k^2 |\lambda|^2 C_{T, \Gamma, \sigma}}{2} < \infty, \label{li1}
\end{align}
where the first inequality follows from H\"{o}lder's inequality, and the third inequality holds because the process $\{e^{k \lambda_i F_{R, t_i}(r) - \frac{k^2 \lambda_i^2}{2}\langle F_{R,t_i} \rangle_{r}}\}_{r \in [0,T]}$ is a supermartingale and we have 
\begin{equation}
    \E\left[e^{k \lambda_i F_{R, t_i}(t_i) - \frac{k^2 \lambda_i^2}{2}\langle F_{R,t_i} \rangle_{t_i}}\right] \leq \E\left[e^{k \lambda_i F_{R, t_i}(0) - \frac{k^2 \lambda_i^2}{2}\langle F_{R,t_i} \rangle_{0}}\right] = 1.
\end{equation}
Now we show the convexity of $I^{\mathbb{T}_k}$ in \eqref{rate function of mu_R^(t)}. By \cite[Theorem 4.1.18]{Dembo-Zeitouni}, we have
\begin{equation}
    -I^{\mathbb{T}_k}(x) 
    = \inf_{\delta > 0,\  y \in \mathbb{B}_k(x, \delta)} \liminf_{R \to \infty}\frac{1}{R}\log \mu_R^{\mathbb{T}_k}(\mathbb{B}_k(y, \delta)) 
    = \inf_{\delta > 0,\  y \in \mathbb{B}_k(x, \delta)} \limsup_{R \to \infty}\frac{1}{R}\log \mu_R^{\mathbb{T}_k}(\mathbb{B}_k(y, \delta)).
\end{equation}
Let us take $x_1, x_2 \in \R$ such that $I^{\mathbb{T}_k}(x_i) < \infty, \ (i=1,2)$. 
Then, for this $x_1, x_2$, it holds 
\begin{equation}
\label{lc1}
    \liminf_{R\to \infty}\frac{1}{R}\log \mu_R^{\mathbb{T}_k}(\mathbb{B}_k(x_i, \delta)) > - \infty \quad \text{for any $\delta > 0$} \quad (i=1,2).
\end{equation}
Because $y \in \mathbb{B}_k(\frac{x_1 + x_2}{2}, \delta)$ implies $\mathbb{B}_k(\frac{x_1 + x_2}{2}, \delta') \subset \mathbb{B}_k(y,\delta)$ for all $\delta' > 0$ small enough, we see that
\begin{align}
    -I^{\mathbb{T}_k}\left(\frac{x_1 + x_2}{2}\right) 
    &= 
    \inf_{\delta > 0, y \in \bB_k\left(\frac{x_1 + x_2}{2}, \delta \right)}\limsup_{R \to \infty}\frac{1}{R}\log \mu_R^{\mathbb{T}_k}\left(\mathbb{B}_k\left(y, \delta \right)\right)\\
    &=
    \inf_{\delta' > 0}\limsup_{R \to \infty}\frac{1}{R}\log \mu_R^{\mathbb{T}_k}\left(\mathbb{B}_k\left(\frac{x_1 + x_2}{2}, \delta'\right)\right)\\
    &\geq 
    \inf_{\delta > 0}\liminf_{R \to \infty}\frac{1}{2R}\log \mu_{2R}^{\mathbb{T}_k}\left(\mathbb{B}_k\left(\frac{x_1 + x_2}{2}, \delta \right)\right). \label{con1}
\end{align}
By \eqref{lc1}, we can apply Lemma \ref{Lem convexity inequality} to obtain
\begin{align}
    &\liminf_{R \to \infty}\frac{1}{2R}\log \mu_{2R}^{\mathbb{T}_k}\left(\mathbb{B}_k\left(\frac{x_1 + x_2}{2}, \delta \right)\right) \nonumber\\ 
    &\geq 
    \frac{1}{2}\liminf_{R \to \infty}\frac{1}{R}\log \mu_{R}^{\mathbb{T}_k}\left(\mathbb{B}_k\left(x_1, \frac{\delta}{4} \right)\right) + \frac{1}{2}\liminf_{R \to \infty}\frac{1}{R}\log \mu_{R}^{\mathbb{T}_k}\left(\mathbb{B}_k\left(x_2, \frac{\delta}{4} \right)\right).  \label{con2}
\end{align}
Combining \eqref{con1} and \eqref{con2}, we obtain
\begin{equation}
\label{convex ineq}
    -I^{\mathbb{T}_k}\left(\frac{x_1 + x_2}{2}\right)
    \geq -\frac{1}{2}I^{\mathbb{T}_k}(x_1) - \frac{1}{2}I^{\mathbb{T}_k}(x_2).
\end{equation}
Since \eqref{convex ineq} trivially holds for $x_i$ such that $I^{\mathbb{T}_k}(x_i) = \infty$, we have \eqref{convex ineq} for all $x_1, x_2 \in \R^k$. 
This is sufficient to show the convexity of $I^{\mathbb{T}_k}$, thanks to its lower semicontinuity.
Finally, by Proposition \ref{Prp convexity and duality}, \eqref{li1} and the convexity of $I^{\mathbb{T}_k}$ imply that $I^{\mathbb{T}_k}$ is the Fenchel-Legendre transform of $\Lambda^{\mathbb{T}_k}$, and this completes the proof.
\end{proof}

\section{Sample path LDP}\label{section Sample path LDP}
In this section, we prove Theorem \ref{Thm main 2}.
We make use of the following estimate.

\begin{Lem}\label{Lem diff est}
For any $0 \leq s < t \leq T$, $R>0$, and $r\geq 0$,
\begin{equation}
\label{diff tail est}
    P\left(\left| \frac{1}{R}F_R(t) - \frac{1}{R}F_R(s) \right| \geq r \right) \leq 2\exp{\left(-\frac{Rr^2}{2C_{T,\Gamma, \sigma}|t-s|}\right)}. 
\end{equation}
In particular, for any $p \in \N$, we have 
\begin{equation*}
    \E\left[\left| \frac{1}{R}F_R(t) - \frac{1}{R}F_R(s) \right|^{2p}\right] \leq 2^{p+1}(p!)C_{T,\Gamma,\sigma}^p R^{-p} |t-s|^{p}.
\end{equation*}

\end{Lem}

\begin{proof}
Let 
\begin{equation}
    F_{R,s,t}(q) = \int_0^q\int_{\R}\left\{ \ind_{[0,t)}(r)(G(t-r)\ast \ind_{[0,R]})(z) - \ind_{[0,s)}(r)(G(s-r)\ast \ind_{[0,R]})(z)\right\}\sigma(u(r,z))W(dr,dz).
\end{equation}
We have $F_{R,s,t}(T) = F_R(t)-F_R(s)$ and 
\begin{align}
    \abra{F_{R,s,t}}_T 
    &= \int_0^T\int_{\R^2}\left\{ \ind_{[0,t)}(r)(G(t-r)\ast \ind_{[0,R]})(z) - \ind_{[0,s)}(r)(G(s-r)\ast \ind_{[0,R]})(z)\right\}\sigma(u(r,z))\Gamma(z-\widetilde{z}) \\
    &\qquad \quad \times\left\{ \ind_{[0,t)}(r)(G(t-r)\ast \ind_{[0,R]})(\widetilde{z}) - \ind_{[0,s)}(r)(G(s-r)\ast \ind_{[0,R]})(\widetilde{z})\right\}\sigma(u(r,\widetilde{z}))dzd\widetilde{z}dr\\
    &\leq \norm{\sigma}_{\infty}^2\norm{\Gamma}_{L^1(\R)}\int_0^T\int_{\R}\left| \ind_{[0,t)}(r)\cF G(t-r)(\xi) - \ind_{[0,s)}(r)\cF G(s-r)(\xi) \right|^2|\cF\ind_{[0,R]}(\xi)|^2d\xi dr,
\end{align}
where $\cF$ denotes the Fourier transform. 
Here, in the last inequality, we used the boundedness of $\sigma$, Young’s convolution inequality, and Plancherel’s theorem, in that order.
It is easy to check that 
\begin{equation}
    \int_0^T \left| \ind_{[0,t)}(r)\cF G(t-r)(\xi) - \ind_{[0,s)}(r)\cF G(s-r)(\xi) \right|^2 dr \leq C_T |t-s|
\end{equation}
holds in both the heat and wave cases. 
See, \textit{e.g.}, the proofs of Proposition 4.1 in \cite{1dSHECLT} and \cite{1dSWECLT}. 
Therefore, we obtain $\abra{F_{R,s,t}}_T \leq C_{T,\Gamma, \sigma}|t-s|R$, and the exponential inequality for martingales now gives \eqref{diff tail est}.
\end{proof}

Lemma \ref{Lem diff est} implies that the process $\{R^{-1}F_R(t)\}_{t\in[0,T]}$ has a continuous modification. 
We write $\mu_R$ for the law on $(C([0,T]), \cB(C([0,T])))$ of a continuous modification of this process.

Let $\cO_{\mathrm{p}}$ (resp. $\cO_{\mathrm{u}}$) denote the topology of pointwise (resp. uniform) convergence on $\R^{[0,T]}$. 
Let $\widetilde{\cO_{\mathrm{p}}}$ (resp. $\widetilde{\cO_{\mathrm{u}}}$) be the subspace topology on $C([0,T])$ induced by $\cO_{\mathrm{p}}$ (resp. $\cO_{\mathrm{u}}$).
We write $\cB(\cO)$ for the Borel $\sigma$-algebra of a topology $\cO$. 
Note that $\mu_R$ on $(C([0,T]), \cB(\widetilde{\cO_{\mathrm{u}}}))$ can naturally be regarded as a measure on $(\R^{[0,T]}, \cB(\cO_{\mathrm{p}}))$ since $\cO_{\mathrm{p}} \subset \cO_{\mathrm{u}}$ and $\cB(\widetilde{\cO_{\mathrm{u}}}) = \cB(\cO_{\mathrm{u}}) \cap C([0,T])$.

The following proposition is a direct consequence of Theorem \ref{Thm main 1} and Proposition \ref{Prp proj gene}. 

\begin{Prp}
\label{Prp projective LDP}
The family of probability measures $\{\mu_R\}_{R>0}$ on $(\R^{[0,T]}, \cB(\cO_{p}))$ satisfies the LDP with speed $R$ and the convex good rate function 
\begin{equation}
    I(f) = \sup_{k \in \N} \sup_{0 \leq t_1 < \cdots < t_k \leq T}\bigg\{ I^{\mathbb{T}_k} \Big( (f(t_1), \ldots, f(t_k)) \Big) \bigg\}, \qquad f \in \R^{[0,T]},
\end{equation}
where $I^{\bT_k}$ is the rate function specified in Theorem \ref{Thm main 1}.
\end{Prp}

In light of Proposition \ref{Prp projective LDP}, the proof of Theorem \ref{Thm main 2} is completed once we establish the exponential tightness of $\{\mu_R\}_{R>0}$ with respect to the finer topology $\cO_{\mathrm{u}}$. 
To this end, we use the following estimate, which is a consequence of the embedding result between Hölder and Orlicz spaces given in \cite[Section 7]{Schiedpaper}.

\begin{Lem}[{\cite[Corollary 7.1]{Schiedpaper}}]
\label{Lem schied corollary 7.1}
Let $n \in \N$ and $\{\xi(t)\}_{t \in [0,T]^n}$ be a real-valued continuous process. 
Suppose that there exists $p, q, r \in (0, \infty)$ such that 
\begin{equation}
    \sup_{s, t \in [0,T]^n, s \neq t}\E\left[ \exp \left( \frac{p}{|t-s|^{q}} |\xi(t) - \xi(s)| \right) \right] \leq r.
\end{equation}
Then, for any $q' \in (0, q)$ and $M \geq 0$, 
\begin{equation}
    P\left(\sup_{s, t \in [0,T]^n, s \neq t} \frac{|\xi(t) - \xi(s)|}{|t-s|^{q'}} \geq M \right) \leq (1 + r)\exp \left( -\frac{Mp}{c_{n,q,q',T}} \right),
\end{equation}
where the constant $c_{n, q, q', T}$ depends only on $n$, $q$, $q'$, and $T$, and can be taken, for instance, as 
\begin{equation}
    c_{n,q,q',T} = \frac{T^{q-q'}(1+Q!)2^{q' + 1 + \frac{n}{Q}}}{1-2^{-(q-q')+ \frac{n}{Q}}}, \qquad \text{with} \quad Q \coloneqq \left\lfloor\frac{n}{q-q'}\right\rfloor + 1.
\end{equation}
\end{Lem}

\begin{Rem}
Although \cite[Section 7]{Schiedpaper} considers continuous processes indexed by $[0,1]^n$, we note for completeness that the same argument applies to the case $[0, T]^n$ without any essential modification.
\end{Rem}

\begin{Prp}
The family of probability measures $\{\mu_R\}_{R>0}$ on $(\R^{[0,T]}, \cB(\mathcal{O}_{\mathrm{u}}))$ is exponentially tight. 
Consequently, $\{\mu_R\}_{R>0}$ on $C([0,T], \cB(\widetilde{\cO_{\mathrm{u}}}))$ satisfies the same LDP as in Proposition \ref{Prp projective LDP}.
\end{Prp}

\begin{proof}
For $M > 0$ and $a>0$, set 
\begin{equation*}
    H^a_{M} \coloneqq \left\{f \in C([0,T]) \relmiddle{|} f(0) = 0, \sup_{s,t \in [0,T], s \neq t} \frac{|f(t) - f(s)|}{|t-s|^a} \leq M \right\}.
\end{equation*}
By the Ascoli-Arzel\`{a} theorem, $H^a_{M}$ is a compact subset of $(C([0,T]), \widetilde{\mathcal{O}_{\mathrm{u}}})$, and hence $\iota(H^a_{M})$ is compact in $(\R^{[0,T]}, \mathcal{O}_{\mathrm{u}})$, where $\iota \colon (C([0,T], \widetilde{\mathcal{O}_{\mathrm{u}}})) \to (\R^{[0,T]}, \mathcal{O}_{\mathrm{u}})$ is the inclusion map. 

Let $0\leq s < t \leq T$ and $A>0$. 
Using $\E[e^{cX}] = c\int_{\R}e^{cr}P(X \geq r)dr$, which holds for any random variable $X$ and $c>0$, and applying \eqref{diff tail est}, we obtain 
\begin{equation}
    \E[e^{\frac{A}{R}|F_R(t) - F_R(s)|}] 
    \leq \sqrt{\frac{8\pi C_{T, \Gamma, \sigma} |t-s|A^2}{R}} \exp\left(\frac{C_{T, \Gamma, \sigma} |t-s|A^2}{2R}\right).
\end{equation}
Thus, taking $A = R|t-s|^{-\frac{1}{2}}$, we have
\begin{equation*}
    \sup_{s,t \in [0,T], s \neq t}\E\left[\exp{\left( \frac{R}{|t-s|^{\frac{1}{2}}} \left| \frac{1}{R}F_R(t) - \frac{1}{R}F_R(s) \right| \right)}\right] \leq \sqrt{8\pi C_{T, \Gamma, \sigma}R} \exp\left(\frac{C_{T, \Gamma, \sigma}R}{2}\right).
\end{equation*}
This, together with Lemma \ref{Lem schied corollary 7.1}, implies that for any $\delta < \frac{1}{2}$, 
\begin{equation}
\label{mu_R exp tight ineq}
    P\left(\sup_{s,t \in [0,T], s \neq t} \frac{1}{|t-s|^{\delta}} \left| \frac{1}{R}F_R(t) - \frac{1}{R}F_R(s) \right| > M \right) \leq \left(1 + \sqrt{8\pi C_{T, \Gamma, \sigma}R} \exp{\left(\frac{C_{T, \Gamma, \sigma}R}{2} \right)}\right) \exp{\left(-\frac{M R}{C_{T, \delta}}\right)}.
\end{equation}
Since the left-hand side of above equals $\mu_R((\iota(H^{\delta}_{M}))^c)$, it follows from \eqref{mu_R exp tight ineq} that 
\begin{equation}
    \limsup_{R \to \infty} \frac{1}{R} \log \mu_R((\iota(H_M^{\delta}))^{c}) \leq \left(-\frac{M}{C_{T, \delta}}\right) \lor  \left(\frac{C_{T, \Gamma, \sigma}}{2} -  \frac{M}{C_{T, \delta}}\right) = \frac{C_{T, \Gamma, \sigma}}{2} -  \frac{M}{C_{T, \delta}}.
\end{equation}
Therefore, we conclude that the family $\{\mu_R\}_{R>0}$ is exponentially tight on $(\R^{[0,T]}, \cB(\mathcal{O}_{\mathrm{u}}))$.
Combining this exponential tightness and Proposition \ref{Prp projective LDP}, we deduce from \cite[Corollary 4.2.6]{Dembo-Zeitouni} that $\{\mu_R\}_{R>0}$ on $(\R^{[0,T]}, \cB(\mathcal{O}_{\mathrm{u}}))$ satisfies the same LDP as in Proposition \ref{Prp projective LDP}.
Therefore, the same LDP also holds in $(C([0,T]), \cB(\widetilde{\mathcal{O}_{\mathrm{u}}}))$ since $C([0,T])$ is a closed subset of $(\R^{[0,T]}, \mathcal{O}_{\mathrm{u}})$ (\textit{cf.} \cite[Lemma 4.1.5]{Dembo-Zeitouni}).
\end{proof}

\bibliographystyle{amsalpha}
\bibliography{main}

\providecommand{\bysame}{\leavevmode\hbox to3em{\hrulefill}\thinspace}
\providecommand{\MR}{\relax\ifhmode\unskip\space\fi MR }
\providecommand{\MRhref}[2]{%
  \href{http://www.ams.org/mathscinet-getitem?mr=#1}{#2}
}
\providecommand{\href}[2]{#2}
\begin{thebibliography}{CKNP21}

\bibitem[AF03]{AdamsFournier}
Robert~A. Adams and John J.~F. Fournier, \emph{Sobolev spaces}, second ed., Pure and Applied Mathematics (Amsterdam), vol. 140, Elsevier/Academic Press, Amsterdam, 2003. \MR{2424078}

\bibitem[BD96]{MixingLDP}
W\l~odzimierz Bryc and Amir Dembo, \emph{Large deviations and strong mixing}, Ann. Inst. H. Poincar\'e{} Probab. Statist. \textbf{32} (1996), no.~4, 549--569. \MR{1411271}

\bibitem[CKNP21]{SHEergo}
Le~Chen, Davar Khoshnevisan, David Nualart, and Fei Pu, \emph{Spatial ergodicity for {SPDE}s via {P}oincar\'e-type inequalities}, Electron. J. Probab. \textbf{26} (2021), Paper No. 140, 37. \MR{4346664}

\bibitem[Dal99]{Dal99}
Robert~C. Dalang, \emph{Extending the martingale measure stochastic integral with applications to spatially homogeneous s.p.d.e.'s}, Electron. J. Probab. \textbf{4} (1999), no. 6, 29. \MR{1684157}

\bibitem[DG87]{Dawson-Gartner}
Donald~A. Dawson and J\"{u}rgen G\"{a}rtner, \emph{Large deviations from the {M}c{K}ean-{V}lasov limit for weakly interacting diffusions}, Stochastics \textbf{20} (1987), no.~4, 247--308. \MR{885876}

\bibitem[DQS11]{DQ}
Robert~C. Dalang and Llu\'is Quer-Sardanyons, \emph{Stochastic integrals for spde's: a comparison}, Expo. Math. \textbf{29} (2011), no.~1, 67--109. \MR{2785545}

\bibitem[DVNZ20]{1dSWECLT}
Francisco Delgado-Vences, David Nualart, and Guangqu Zheng, \emph{A central limit theorem for the stochastic wave equation with fractional noise}, Ann. Inst. Henri Poincar\'{e} Probab. Stat. \textbf{56} (2020), no.~4, 3020--3042. \MR{4164864}

\bibitem[DZ10]{Dembo-Zeitouni}
Amir Dembo and Ofer Zeitouni, \emph{Large deviations techniques and applications}, Stochastic Modelling and Applied Probability, vol.~38, Springer-Verlag, Berlin, 2010, Corrected reprint of the second (1998) edition. \MR{2571413}

\bibitem[Ebi24]{3dSWECLT}
Masahisa Ebina, \emph{Central limit theorems for nonlinear stochastic wave equations in dimension three}, Stoch. Partial Differ. Equ. Anal. Comput. \textbf{12} (2024), no.~2, 1141--1200. \MR{4734622}

\bibitem[Ebi25]{Ebi25}
\bysame, \emph{Central limit theorems for stochastic wave equations in high dimensions}, Electron. J. Probab. \textbf{30} (2025), Paper No. 60, 56. \MR{4890842}

\bibitem[Ham62]{Hammersley}
J.~M. Hammersley, \emph{Generalization of the fundamental theorem on sub-additive functions}, Proc. Cambridge Philos. Soc. \textbf{58} (1962), 235--238. \MR{137800}

\bibitem[HNV20]{1dSHECLT}
Jingyu Huang, David Nualart, and Lauri Viitasaari, \emph{A central limit theorem for the stochastic heat equation}, Stochastic Process. Appl. \textbf{130} (2020), no.~12, 7170--7184. \MR{4167203}

\bibitem[HNVZ20]{SHECLT}
Jingyu Huang, David Nualart, Lauri Viitasaari, and Guangqu Zheng, \emph{Gaussian fluctuations for the stochastic heat equation with colored noise}, Stoch. Partial Differ. Equ. Anal. Comput. \textbf{8} (2020), no.~2, 402--421. \MR{4098872}

\bibitem[LZ23]{1dSHELIL}
Jingyu Li and Yong Zhang, \emph{The law of the iterated logarithm for spatial averages of the stochastic heat equation}, Acta Math. Sci. Ser. B (Engl. Ed.) \textbf{43} (2023), no.~2, 907--918. \MR{4535893}

\bibitem[Mue91]{MC2}
Carl Mueller, \emph{Long time existence for the heat equation with a noise term}, Probab. Theory Related Fields \textbf{90} (1991), no.~4, 505--517. \MR{1135557}

\bibitem[Nua06]{Nualart_book}
David Nualart, \emph{The {M}alliavin calculus and related topics}, second ed., Probability and its Applications (New York), Springer-Verlag, Berlin, 2006. \MR{2200233}

\bibitem[NZ20]{d123SWEergo}
David Nualart and Guangqu Zheng, \emph{Spatial ergodicity of stochastic wave equations in dimensions 1, 2 and 3}, Electron. Commun. Probab. \textbf{25} (2020), Paper No. 80, 11. \MR{4187721}

\bibitem[NZ22]{SWE12d}
\bysame, \emph{Central limit theorems for stochastic wave equations in dimensions one and two}, Stoch. Partial Differ. Equ. Anal. Comput. \textbf{10} (2022), no.~2, 392--418. \MR{4439987}

\bibitem[Sch97]{Schiedpaper}
Alexander Schied, \emph{Moderate deviations and functional {LIL} for super-{B}rownian motion}, Stochastic Process. Appl. \textbf{72} (1997), no.~1, 11--25. \MR{1483609}

\bibitem[Wal86]{Walsh}
John~B. Walsh, \emph{An introduction to stochastic partial differential equations}, \'Ecole d'\'et\'e{} de probabilit\'es de {S}aint-{F}lour, {XIV}---1984, Lecture Notes in Math., vol. 1180, Springer, Berlin, 1986, pp.~265--439. \MR{876085}

\end{thebibliography}

\end{document}